\documentclass[12pt]{article}

\usepackage{amssymb,amsthm,amsmath}
\usepackage[usenames,dvipsnames]{color}
\usepackage{soul}
\usepackage{longtable}
\allowdisplaybreaks

\newcommand{\comment}[1]{}
\definecolor{teal}{RGB}{0,128,128}
\definecolor{darkpurple}{RGB}{128,0,128}

\usepackage[titletoc,toc,title]{appendix}

\newcommand{\pnote}[1]{{\color{teal}{\sf #1}}}

\newtheorem{theorem}{Theorem}[section]

\newtheorem{lemma}[theorem]{Lemma}
\newtheorem{cor}[theorem]{Corollary}

\theoremstyle{definition}
\newtheorem{defn}[theorem]{Definition}

\def \cC {{\cal C}}

\def \cH {{\cal H}}

\def \Z {\mathbb Z}

\title {On the Hamilton-Waterloo Problem with odd orders}

\author{A.\ C.\ Burgess \footnotemark[1] 
\and 
P.\ Danziger \footnotemark[3] 
\footnotemark[4]
\and
T.\ Traetta \footnotemark[3]
}

\begin{document}

\maketitle

\footnotetext[1]{Department of Mathematical Sciences, University of New Brunswick, 100 Tucker Park Rd., Saint John, NB  E2L 4L5, Canada}
\footnotetext[3]{Department of Mathematics, Ryerson University, 350 Victoria St., Toronto, ON  M5B 2K3, Canada}
\footnotetext[4]{Corresponding author, danziger@ryerson.ca.}

\begin{abstract}
Given non-negative integers $v, m, n, \alpha, \beta$, the Hamilton-Waterloo problem asks for a factorization of the complete graph $K_v$ into $\alpha$  $C_m$-factors and $\beta$ $C_n$-factors.
Clearly, $v$ odd, $n,m\geq 3$, $m\mid v$, $n\mid v$ and $\alpha+\beta = (v-1)/2$ are necessary conditions. 
To date results have only been found for specific values of $m$ and $n$.
In this paper we show that for any $m$ and $n$ the necessary conditions are sufficient when $v$ is a multiple of $mn$ and $v>mn$, except possibly when $\beta=1$ or 3, with five additional possible exceptions in $(m,n,\beta)$.
For the case where $v=mn$ we show sufficiency when $\beta > (n+5)/2$ except possibly when $(m,\alpha) = (3,2)$, $(3,4)$, with seven further possible exceptions in $(m,n,\alpha,\beta)$.
We also show that when $n\geq m\geq 3$ are odd integers, the lexicographic product of $C_m$ with the empty graph of order $n$ has a factorization into $\alpha$ $C_m$-factors and $\beta$ $C_n$-factors for every $0\leq \alpha \leq n$, $\beta = n-\alpha$, except possibly when $\alpha= 2,4$, $\beta = 1, 3$, with three additional possible exceptions in $(m,n,\alpha)$.
\end{abstract}

Keywords: 2-Factorizations, Resolvable Cycle Decompositions, Cycle Systems, Generalized Oberwolfach Problem, Hamilton-Waterloo Problem.

\section{Introduction}

We assume that the reader is familiar with the general concepts of graph theory and design theory, and refer them to \cite{Handbook, West}. In particular, a factor of a graph $G$ is a spanning subgraph of $G$; a 1-factor is a factor which is 1-regular and a 2-factor is a factor which is 2-regular and hence consists of a collection of cycles. We denote cycles of length $n$ by $C_n$, with consecutive points $(x_0,x_1\ldots, x_{n-1})$,
and a collection of cycles by $[n_1^{a_1}, \ldots, n_s^{a_s}]$, where there are $a_i$ cycles of length $n_i$.
We denote the complete graph on $n$ vertices by $K_n$. By $K_n^*$ we mean the graph $K_n$ when $n$ is odd and $K_n-I$, where $I$ is a single 1-factor, when $n$ is even.

A 2-factorization of a graph, $G$, is a partition of the edges of $G$ into 2-factors. It is well known that a regular graph has a 2-factorization if and only if every vertex has even degree; see~\cite{Rosa}.
However, if we specify a particular 2-factor, $F$ say, and ask for all the factors to be isomorphic to $F$ the problem becomes much harder. Indeed, if $G\cong K_n^*$, we have the Oberwolfach problem, which is well known to be hard. For a fixed $2$-factor $F$ of $K_n^*$, we denote the Oberwolfach problem by OP$(F)$. 
More generally, given a graph $G$ and a collection of graphs $\cH$, an {\em $\cH$-factor of $G$} is a set of edge-disjoint subgraphs of $G$, each isomorphic to a member of $\cH$, which between them cover every point in $G$. An {\em $\cH$-factorization of $G$} is a set of edge-disjoint $\cH$-factors of $G$. If $\cH$ consists of a single graph, $H$, we speak of $H$-factors and $H$-factorizations of $G$ respectively. 

A 2-factor is a Hamiltonian cycle if it contains only one cycle. We will use the following classical result on Hamiltonian factorizations, originally due to Walecki \cite{Lucas}; see also \cite[Section VI.12]{Handbook}.
\begin{theorem}
\label{Hamiltonian}
For any $v \geq 3$ there exists a factorization of $K_v^*$ into Hamiltonian cycles.
\end{theorem}
We call a factor in which every component of the factor is isomorphic a {\em uniform} factor. 
The case where $F$ is uniform has also been solved \cite{Alspach Haggvist 85, ASSW, Hoffman Schellenberg 91}.
\begin{theorem}
\label{OP uniform}
Let $v$ and $n$ be integers.
There is a $C_n$-factorization of $K_v^*$ if and only if $n \mid v$, except that there is no $C_3$-factorization of $K_{6}^*$ or $K_{12}^*$.
\end{theorem}
It is known that the Oberwolfach problem has no solution when $F \in \{[3^2 ], [3^4]$, $[4, 5], [3^2 , 5]\}$. Otherwise, a solution is known for every case where $n\leq 40$ \cite{DFWMR 10}.
In the case when $F$ is bipartite, and so contains only even cycles, OP$(F)$ is solved \cite{Alspach Haggvist 85, Bryant Danziger 11, Haggvist 85}. OP$(F)$ is also solved in the case where $F$ has exactly two cycles 
\cite{BBR97, Buratti Traetta 12, Traetta 13}. In addition, cyclic solutions have also been studied \cite{Buratti Rinaldi 05, Buratti Rinaldi 08, Buratti Traetta 12} and many other families are known \cite{Bryant Schar 09, Ollis Sterr 09}, but no general solution is known. See \cite[Section VI.12]{Handbook} for a survey. 

We use the following notation from \cite{OCD, Liu 00} (which is adapted from \cite[p.\ 393]{West}).
\begin{defn}
Given a graph $G$, $G[n]$ is the {\em lexicographic product} of $G$ with the empty graph on $n$ points. Specifically, the point set is $V(G)\times \Z_n$ and $(x,i)(y,j)\in E(G[n])$ if and only if $xy \in E(G)$, $i,j\in \Z_n$.
\end{defn} 
We note that $G[m][n] \cong G[mn]$.

The existence of 2-factorizations of other graphs has also been considered. In particular,
Liu \cite{Liu 03, Liu 00} considered $C_m$-factorizations of the complete multipartite graph and showed the following result. 
\begin{theorem}[\cite{Liu 03}]
\label{Liu}
There exists a $C_m$-factorization of $K_t[n]$ if and only if $m\mid tn$, $(t-1)n$ is even, further $m$ is even when $t = 2$, and
$(m, t, n) \not\in \{(3, 3, 2), (3, 6, 2), (3, 3, 6), (6, 2, 6)\}$.
\end{theorem}

A well-known variant of the Oberwolfach problem is the Hamilton-Waterloo problem. In this case we are given two specified 2-factors, $F$ and $F'$, and two integers, $\alpha$ and $\beta$, and asked to factor a graph $G$ into $\alpha$ factors isomorphic to $F$ and $\beta$ factors isomorphic to $F'$. We denote a solution to this problem by HW$(G;F,F';\alpha,\beta)$. We denote by HWP$(G; F, F')$ the set of $(\alpha, \beta)$ for which a solution HW$(G;F,F';\alpha,\beta)$ exists.
In the case where $F$ and $F'$ are uniform, with cycle lengths $m$ and $n$ say, we refer to HW$(G;m,n;\alpha,\beta)$ and  HWP$(G; m, n)$ as appropriate. Further, if $G=K_v$, we refer to HW$(v;m,n;\alpha,\beta)$ and HWP$(v; m, n)$ respectively. 
We note the following necessary condition for the case of uniform factors.

\begin{theorem}
\label{nec}
Given odd integers $m$ and $n$, in order for HW$(mnt;m,n;\alpha,\beta)$ to exist, $t$ must be odd and $\alpha+\beta = (mnt-1)/2$.
\end{theorem}

We note that when $\gcd(m,n)\neq1$, it is possible to have solutions where the number of points, $v$, is not a multiple of $mn$, though we must have $v$ odd, $m\mid v$ and $n\mid v$.  

The first serious analysis of the Hamilton-Waterloo problem in the literature appeared in 2003 \cite{AdaBilBryElz}, but the problem was known well before then. It is known that the following instances of the Hamilton-Waterloo problem do not exist.\\[1ex]
\begin{tabular}{l}
$\mathrm{HW}(7;[3,4],[7];2,1)$, $\mathrm{HW}(9;[3^3],[4,5];2,2)$ \\
HW$(9; [3^3],F;3,1)$ for $F\in\{[4,5],[3,6],[9]\}$ and  \\
HW$(15; [3^5],F;6,1)$ for $F\in\{[3^2,4,5],[3,5,7],[5^3],[4^2,7],[7,8]\}$.
\end{tabular} \\[1ex]
Every other instance of the Hamilton-Waterloo Problem for which the obvious necessary conditions has a solution when the number of points $v$ is odd and $v\leq 17$ \cite{AdaBry,FraHolRos,FraRos}. In the case where $v$ is even, there is no solution to $\mathrm{HW}(K_6^*;[3^2],F;2,0)$. It is also known that there is no solution to $\mathrm{HW}(K_8^*;[3,5],[4^2];2,1)$ or $\mathrm{HW}(K_8^*;[3,5],[4^2];1,2)$, but that there is a solution in every other instance where the obvious necessary conditions are satisfied and $v\leq 10$, see~\cite[Section VI.12.4]{AdaBry,Handbook}.

In the bipartite case, where all the cycles in $F$ and $F'$ are of even length, $(r,r')\in {\rm HWP}(K_v^*;F,F')$ for all $r,r'>1$ \cite{Bryant Danziger 11}.  When $G$ is a regular complete multipartite graph of even degree and $F$ and $F'$ are bipartite factors of $G$, HWP$(G;F,F')$ has been almost completely  solved \cite{BDP}.  For the case of $C_3$- and $C_4$-factors, $\mathrm{HWP}(K_v^*,3,4)$ is almost completely solved \cite{DQS}. 
Much recent work has focused on the case in which one of the factors is a Hamiltonian cycle~\cite{SM}.
For the even case, solutions can be found with 4-cycles and a single factor of $v$-cycles in \cite{KO, Lei Fu Shen, OO}.
The case of Hamilton cycles and triangle-factors, HWP$(K_v^*;3,v)$, has been considered in \cite{DL,DL2,HNR} for odd $v$ and \cite{Lei Shen} for even $v$, but still remains open.

It is clear that finding solutions of the Hamilton-Waterloo Problem for uniform odd-cycle factors is challenging. Cyclic solutions for sparse families can be found in \cite{Buratti Rinaldi 05,Buratti Danziger 15,Traetta Merola}.
A complete solution for the existence of $\mathrm{HW}(v;m,n;\alpha,\beta)$ in the cases $(m,n)=(3,5), (3,15)$ and $(5,15)$ is given in \cite{AdaBilBryElz}, except that $(6,1) \not\in {\rm HWP}(15; 3, 5)$ and the existence of HW$(v; 3,5;r,1)$ remains open for $v > 15$.
A complete existence result for $(m,n) = (3,7)$
is given in \cite{Lei Fu}. 
The case where $(m,n) = (3,9)$ was solved in
\cite{Kamin}, except possibly when $\beta=1$, this is extended to $(m,n)=(3,3x)$ in \cite{AKKPO}. 
To date these are the only cases known. 

In this paper we consider the case of uniform odd factors in the Hamilton-Waterloo problem. We prove the following theorem.
\begin{theorem}
\label{main}
If $m$ and $n$ are odd integers with $n\geq m \geq 3$, then $(\alpha, \beta)\in {\rm HWP}(mnt; m, n)$ if and only if $t$ is odd, $\alpha, \beta \geq 0$ and $\alpha+\beta = (mnt-1)/2$, except possibly when:
\begin{itemize}
\item
$t>1$ and $\beta = 1$ or $3$, or $(m,n,\beta)= (5,9,5)$, $(5,9,7)$, $(7,9,5)$, $(7,9,7)$, $(3,13,5)$ or;
\item
$t=1$ and $\beta \in [1, \ldots, \frac{n-3}{2}] \cup \left\{\frac{n+1}{2}, \frac{n+5}{2}\right\}$, $(m,\alpha) = (3,2)$, $(3,4)$, or $(m,n,\alpha,\beta)=  (3,11,6,10)$, $(3,13,8,10)$, $(5,7,9,8)$, $(5,9,13,9)$, or $(7,9,22,9)$.
\end{itemize}
\end{theorem}

On the way, we also prove the following useful result on factorizations of $C_m[n]$.
\begin{theorem}
\label{C_m[n]}
If $n$, $m$ and $\alpha$ are odd integers with $n\geq m\geq 3$, $0\leq \alpha \leq n$, then $(\alpha, \beta)\in {\rm HWP}(C_m[n]; m, n)$, if and only if $\beta=n-\alpha$, except possibly when $\alpha= 2,4$, $\beta = 1, 3$, or $(m,n,\alpha) = (3,11,6)$, $(3,13,8)$, $(3,15,8)$.
\end{theorem}

In the next section we introduce some tools and provide some powerful methods which we use in Section~\ref{Section C_m[n]} where we prove Theorem~\ref{C_m[n]}. In Section~\ref{Section Main} we prove the main result, Theorem~\ref{main}.

\section{Preliminaries}
\label{Section prelim}

In this section we state some known results and develop the tools we will need for the 2-factorizations.
We use $[a,b]$ to denote the set of integers from $a$ to $b$ inclusive.

\subsection{Langford and Skolem Sequences}

We will make extensive use of Skolem and Langford sequences, see \cite[Section VI.53]{Handbook}.

\begin{defn}[Skolem and Langford sequence]
Let $d$ and $\nu$  be positive integers, and let $1\leq k \leq 2\nu+1$.
A {\it $k$-extended Langford sequence} of {\it order $\nu$} and {\it defect $d$}, is a sequence of $\nu$ integers 
$(\ell_1, \ell_2, \ldots, \ell_\nu)$ such that
\[\{\ell_i, \ell_i +i + (d-1) \;|\; i=1, \ldots, \nu\} = \{1, 2, \ldots, 2\nu+1\}\setminus \{k\}.\] 
An extended Langford sequence of order $\nu$ with $k=2\nu+1$ is simply referred to as a {\em Langford sequence}, while
if $k=2\nu$, the sequence is usually called a \emph{hooked} Langford sequence.
When $d=1$ and $k=2\nu+1$ or $2\nu$, we obtain a {\it Skolem sequence} or {\it hooked Skolem sequence}, respectively. 
\end{defn}

We state the following existence results for Skolem and Langford sequences.
\begin{theorem}[\cite{Okeefe, Skolem}]\rule{0ex}{1ex}
\label{Skolem}
\begin{itemize}
\item
A Skolem sequence of order $\nu$ exists if and only if $\nu\equiv 0, 1 \pmod{4}$.
\item
A hooked Skolem sequence of order $\nu$ exists if and only if $\nu\equiv 2, 3 \pmod{4}$.
\end{itemize}
\end{theorem}

\begin{theorem}[\cite{Simpson}]
\label{lang}
A Langford sequence of order $\nu$ and defect $d$ exists if and only if
\begin{enumerate}
  \item $\nu\geq 2d-1$, and
  \item $\nu\equiv 0,1 \pmod{4}$ and $d$ is odd, or $\nu\equiv 0,3 \pmod{4}$ and $d$ is even.
\end{enumerate}
A hooked Langford sequence of order $\nu$ and defect $d$ exists if and only if
\begin{enumerate}
  \item $\nu(\nu-2d+1) + 2 \geq 0$, and
  \item $\nu\equiv 2,3 \pmod{4}$ and $d$ is odd, or $\nu\equiv 1,2 \pmod{4}$ and $d$ is even.
\end{enumerate}
\end{theorem}

We will frequently use the well known result that given a Langford sequence of order $\nu$ and defect $d$, $L=(\ell_1, \ell_2, \ldots, \ell_\nu)$, then the $\nu$ triples $T_i=\{t_{i1}, t_{i2}, t_{i3}\}$, where  $t_{i1} = i + (d - 1)$, $t_{i2} = \ell_i + \nu + (d-1)$ and $t_{i3} = \ell_i + i + \nu + 2(d-1)$, for $i=1, \ldots, \nu$, satisfy the following properties:
\begin{enumerate}
\item $\{T_i, -T_i \mid i=1,\ldots, \nu\}$ is a partition of the set 
$\pm[d, d+3\nu -1]$;
\item $t_{i1} + t_{i2} + t_{i3}=0$ for $i=1,2,\ldots, \nu$.
\end{enumerate}
These triples are said to be {\em induced by the Langford sequence.}  If we use a hooked Langford sequence, $\{T_i, -T_i \mid i=1,\ldots, \nu\}$ partition the set $\pm[d, d+3\nu]\setminus\{\pm(d+ 3\nu-1)\}$ instead.

\subsection{Factoring Cayley Graphs}

We will make extensive use of 2-factorizations of $C_m[n]$. 
We take the point set of $C_m[n]$ to be $\Z_m\times\Z_n$ and denote the points by $x_i$, where $x\in \Z_m$, $i\in \Z_n$.  
When we are talking about an individual cycle we may reorder the points so that the original $m$-cycle is $(0,1\ldots, m-1)$. 
Given a set of cycle factors, $\cC$, of $C_m[n]$, and a set  $S\subseteq \Z_n$ we say that $\cC$ {\em covers} $S$ if for every $d\in S$, the edges of the form $x_i(x+1)_{i+d}$ appear in some cycles of $\cC$.
We say that $\cC$ {\em exactly covers} $S$ if every edge in $\cC$ is of this form.

We will make use of the notion of a Cayley graph on a group $\Gamma$. 
Given $S\subseteq \Gamma$, the {\em Cayley Graph} cay$(\Gamma, S)$ is a graph with vertex set $G$ and edge set $\{a,a+d \mid d\in S\}$.
When $\Gamma = \Z_n$ this graph is known as a {\em circulant graph} and denoted $\langle S\rangle_n$. We note that the edges generated by $d\in S$ are the same as those generated by $-d\in -S$, so that cay$(\Gamma,S) = {\rm cay}(\Gamma, \pm S)$, and that the degree of each point is $|S\cup (-S)|$. 
We will use the following Theorem due to Bermond, Favaron, and Mah\'{e}o, \cite{Bermond}.
\begin{theorem}{\cite{Bermond}}
\label{berm}
Every connected 4-regular Cayley graph on a finite Abelian group has a Hamilton factorization.
\end{theorem}

We extend the notion of Cayley graphs to $C_m[n]$.
Given an $m$-cycle $C_m = (0,1\ldots, m-1)$ with an orientation, which we take to be ascending, and a set $S\subseteq\Z_n$ we define the Cayley graph $C_m[S]$ to be the graph with point set $\mathbb{Z}_m\times \Z_n$ and edges $i_{x}(i+1)_{x+d}$, $i\in \mathbb{Z}_m$, $x\in\Z_n$ and $d\in S$. We call $d$ a {\em (mixed) difference}. Note that $C_m[S] = {\rm cay}(\Z_m\times \Z_n, \{1\}\times S)$ and that the degree of each point in this graph is $2|S|$.
A set of factors ${\cal C}$ of $C_m[S]$ is said to (exactly) {\em cover} $C_m[S]$ if $\cal C$ (exactly) covers $S$.

In \cite{ASSW}, Alspach, Schellenberg, Stinson and Wagner showed that there exists a $C_p$-factorization of $C_m[p]$ when $p$ is an odd prime. In fact, many of their results hold for any odd integer. They define the notion of an $i$-projection, which we also extend to a reverse $i$-projection.
\begin{defn}[(reverse) $i$-projection]
Let $n$ and $m$ be odd integers with $n\geq m\geq 3$. Given a directed Hamiltonian cycle, $H = (x_0, x_1,\ldots, x_{n-1})$, of $K_n$, the {\em $i$-projection} of $H$ is the $n$-cycle in $C_m[n]$ defined by 
$$(i_{x_0}, (i+1)_{x_1}, \ldots,(i+m-1)_{x_{m-1}}, i_{x_m}, (i+1)_{x_{m+1}}, i_{x_{m+2}}, (i+1)_{x_{m+3}}, \ldots, (i+1)_{x_{n-1}}).$$
The {\em reverse $i$-projection} of $H$ is the $n$ cycle in $C_m[n]$ defined by 
$$(i_{x_0}, (i-1)_{x_1}, \ldots,(i-m+1)_{x_{m-1}}, i_{x_m}, (i-1)_{x_{m+1}}, i_{x_{m+2}}, (i-1)_{x_{m+3}}, \ldots, (i-1)_{x_{n-1}}).$$
\end{defn}
In \cite{ASSW} the following lemma is proved. The extension to the case of reverse cycles is straightforward. 
\begin{lemma}[\cite{ASSW}]
\label{ASSW lemma}
If  $H$ is a directed Hamiltonian cycle of $K_n$, then the $i$-projections of $H$ yield a $C_n$-factor of $C_m[n]$, $\cC$ say, and if the edge $i_x,(i+1)_y$ appears in $C\in\cC$, then the edges $j_x,(j+1)_y$ appear in some cycle of $\cC$ for every $j\in \Z_m$.
Similarly, the reverse $i$-projections of $H$ yield a $C_n$-factor of $C_m[n]$, $\cC'$ say, and if the edge $i_x,(i-1)_y$ appears in $C\in\cC'$, then the edges $j_x,(j-1)_y$ appear in $\cC'$ for every $j\in \Z_m$.
\end{lemma}
From this we may conclude the following lemma.
\begin{lemma}
\label{nFactorHam}
If $\langle S\rangle_n$ 
decomposes into $t$ Hamiltonian factors, then there exist $2t$ $C_n$-factors which exactly cover $C_m[\pm S]$.
\end{lemma}
\begin{proof}
Let the Hamiltonian factors of $G$ be $H_1,\ldots,H_t$ and choose an orientation for each factor. For each factor we take the $i$-projection and the reverse $i$-projection for $0\leq i< m$ and apply Lemma \ref{ASSW lemma}.
\end{proof}

\begin{lemma}
\label{2 factor Lemma}
Let $n\geq m\geq 3$ be odd integers and let $0<d_1,d_2<n$. 
If any linear combination of $d_1$ and $d_2$ is coprime to $n$,
then there exist four $C_n$-factors of $C_m[n]$ which between them exactly cover the differences $\pm d_j$, $j=1,2$.
\end{lemma}
\begin{proof}
If  $ad_1+bd_2$ is coprime to $n$, for some $a$ and $b$, then the group generated by $d_1$ and $d_2$ is $\Z_n$ and the 
the 4-regular circulant graph $\langle\, \{d_1,d_2\}\,\rangle_n$ is connected. 
Thus by Theorem~\ref{berm} $\langle\, \{d_1,d_2\}\,\rangle_n$ has a Hamiltonian factorization. Lemma~\ref{nFactorHam} then gives the result.
\end{proof}
In particular, we note that the conditions of Lemma~\ref{2 factor Lemma} will be satisfied whenever either $d_1$, $d_2$ or their difference $d_1-d_2$ is coprime to $n$.

\begin{lemma}
\label{1 factor Lemma}
Let $n\geq m \geq 3$ be odd integers and let $0<d<n$ be coprime to $n$. There exist two $C_n$-factors of $C_m[n]$ which between them exactly cover the differences $\pm d$.
\end{lemma}
\begin{proof}
Create the 2-regular circulant graph of $K_n$ using the difference $d$, $\langle \,\{d\}\, \rangle_n$;  this graph is a Hamiltonian cycle. Applying Lemma~\ref{nFactorHam} gives the result.
\end{proof}

The proof of Theorem 5 of \cite{ASSW} shows the following result for $b=1$.
In that theorem the result is only claimed for $n$ prime, but it is straightforward to see that it holds for any odd integer $n$. 
We note that, since $n$ is odd, any power of 2 will be coprime to $n$, though in practice we will use this theorem with $b=1$, except when explicitly noted otherwise.
\begin{theorem}[\cite{ASSW}]
\label{5 classes}
If $n\geq m\geq 3$ are odd integers and $b\in\Z_n$ is coprime to $n$, then there exist five $C_n$-factors of $C_m[n]$ which between them exactly cover $C_m[\pm\{0, b, 2b\}]$.
\end{theorem}
\begin{proof}
The case $b=1$ is given in the proof of Theorem 5 of \cite{ASSW}. For $b\geq 1$, we note that multiplication by $b$ is an automorphism of $\Z_n$.
\end{proof}

\begin{theorem}
\label{constructionD}
Let $T$ be a subset of $\Z_n$. If there exists a  $|T| \times \ell$ matrix $A=[a_{ij}]$ with entries from $T$ such that:
\begin{enumerate}
    \item 
    \label{constructionD cond 1}
    each row of $A$ sums to $0$,
    \item
    \label{constructionD cond 2}
	 each column of $A$ is a permutation of $T$,
\end{enumerate}
then there exists a $C_\ell$-factorization of $C_\ell[T]$. Moreover, if we also have that:
\begin{enumerate}
\setcounter{enumi}{2}
    \item 
    \label{constructionD cond 3}
	$T$ is closed under taking negatives,
\end{enumerate}  
  then there is a $C_m$-factorization of $C_m[T]$ for any $m \geq \ell$ with $m\equiv \ell \pmod{2}$.
\end{theorem}
\begin{proof}
For $1\leq i,h \leq \ell$, we set $s_{i,h} = \sum_{j=1}^{h} a_{ij}$. Note that, by assumption, 
$s_{i,\ell}=0$ for any $i$; also,  $s_{i,1}=a_{i1}$ and $s_{i,h+1}-s_{i,h} = a_{i,h+1}$ for any $i,h$,
where it is understood that $s_{i,\ell+1}=s_{i,1}$ and $a_{i,1}=a_{i,\ell+1}$.
For $i=1,2, \ldots, |T|$, we define the $\ell$-cycles $C_i$ as follows:
\[
C_i = (1_{s_{i1}}, 2_{s_{i2}}, \ldots, \ell_{s_{i,\ell}}),
\]
and we denote by $F_{i}$ the $C_\ell$-factor we obtain by developing $C_i$ on the subscripts, namely,
$F_{i} = \{C_{i}+x\mid x \in \Z_n\}$,
 where $C_{i}+x$  is obtained from $C_i$ by replacing each vertex $h_k$ with $h_{k+x}$.

We claim that $\mathcal{F}= \{F_i \;|\; i=1,2, \ldots, |T|\}$ is a $C_\ell$-factorization of $C_\ell[T]$. 
First recall that each edge of $C_{\ell}[T]$ has the form $[h_k, (h+1)_{k+t}]$ for some 
$(h,k)\in \Z_\ell \times \Z_n$ and $t \in T$. Since, by assumption, any column of $A=[a_{ij}]$ is a 
permutation of $T$, there is an integer $i$ such that
$a_{i,h+1}=t$ for any $t\in T$. Note that
$[h_{s_{i,h}}, (h+1)_{s_{i,h+1}}]\in C_i$; also, $s_{i,h+1} = s_{i,h} + a_{i,h+1} = s_{i,h} + t$.
Therefore, $[h_{k}, (h+1)_{k+t}]\in C_{i-s_{i,h}}+k$ and the result follows.

In order to prove the second part we only need show that there exists 
a $|T|\times m$ matrix with entries in $T$ satisfying conditions 1 and 2 for any 
$m\equiv \ell \pmod{2}$ and $m \geq \ell$. Let $m= \ell+ 2q$ and $T=\{t_1,t_2,\ldots, t_{|T|}\}$.  We create the matrix
\[
  A'=\left[
  \begin{array}{ccccc}
    t_1 & -t_1 & \ldots & t_1 & -t_1 \\
    t_2 & -t_2 & \ldots & t_2 & -t_2 \\
    \vdots &   \vdots & \ldots & \vdots & \vdots    \\
    t_{|T|} & -t_{|T|} & \ldots & t_{|T|} & -t_{|T|} 
  \end{array}
  \right].
\]
It is easy to check that the matrix $[A\;\; A']$ is a $|T|\times m$ matrix satisfying conditions 1 and 2, and this completes the proof.
\end{proof}
\begin{theorem}
\label{corD}
Let $n\geq 3$ be an odd integer and let $S$ be a subset of $\Z_{n}\setminus\{0\}$ which is closed under taking negatives. If there exists a partition of $S$ into subsets of odd sizes, 
  $T_1,-T_1, T_2, -T_2, \ldots, T_{u}, -T_{u}$, such that
  \[
    \sum_{t \in T_i} t = 0\;\;\;\text{ for any } i=1,2,\ldots, u,
  \]
then there exists a $C_m$-factorization of $C_m[S]$ for any odd $m\geq max\{|T_1|,$ $|T_2|,$ $\ldots, |T_u|\}$. 
\end{theorem}
\begin{proof}
For any $i=1,2,\ldots,u$, let $\ell_i=|T_i|$ and $T_i=\{t_{i1}, t_{12}, \ldots, t_{i,\ell_i}\}$ and create the matrix 
the $2\ell_i\times\ell_i$ matrix defined as follows:
\[
  A_i=\left[
  \begin{array}{r}
     B_i \\
    -B_i
  \end{array}
  \right], \mbox{ where } 
  B_i=\left[
  \begin{array}{cccccc}
    t_{i1} & t_{i2} & t_{i3}       & \ldots & t_{i,\ell_i-1} & t_{i,\ell_i} \\
    t_{i2} & t_{i3} & t_{i4}       & \ldots & t_{i,\ell_i}   & t_{i1} \\
    \vdots &   \vdots & \vdots     & \ldots & \vdots  & \vdots  \\
    t_{i,\ell_i} & t_{i1} & t_{i2} & \ldots & t_{i,\ell_i-2} & t_{i,\ell_{i}-1} 
  \end{array}
  \right].
\]
In other words, the block $B_i$ is the matrix whose rows are obtained by cyclically permuting the sequence of integers in $T_i$. 
Note that each row of $B_i$, by assumption, sums to $0$ and thus the same property is satisfied by $A_i$. In addition,
$A_i$ is a matrix whose columns are permutations of $T_i \ \cup \ -T_i$. Therefore, $A_i$ satisfies conditions 1, 2, and
3 of Lemma~\ref{constructionD}.  Hence, there exists a $C_m$-factorization of $C_m[\pm T_i]$
for any $m \geq \ell_i$. Since $C_m[S]$ is the edge-disjoint union of the $C_m[\pm T_i]$'s, then the assertion follows.
\end{proof}

\section{Factoring $C_m[n]$}
\label{Section C_m[n]}

In this section we factor $C_m[n]$ into $C_m$- and $C_n$-factors, i.e.\ we give a solution to HW$(C_m[n];m,n;\alpha,\beta)$. 
We first note that a simple counting argument yields the following necessary condition.
\begin{theorem}
Let $n, m\geq 3$ be integers. If a HW$(C_m[n];m,n;\alpha,\beta)$ exists, then $\alpha,\beta \geq 0$ and $\alpha+\beta=n$. 
\end{theorem}
To show sufficiency, we consider cases of $\alpha$ modulo 6. We will find the following lemma useful.

\begin{lemma}
\label{n factor}
Let $n\geq m\geq 3$ be odd integers, let $0\leq w \leq \frac{n-1}{2}$, and let $S \subseteq\Z_n$. Then $C_m[\pm S]$ has a factorization into $2|S|$ $C_n$-factors for
\begin{enumerate}
 \item $S=\left[w, \frac{n-1}{2}\right]$, or
 \item $S = 
\{w\}\cup\left[w+2,\frac{n-1}{2}\right]$, except possibly when $n=2w+3$ and $n\equiv 0\equiv w \pmod{3}$.
\end{enumerate} 
\end{lemma}
\begin{proof}
When $S = \left[w, \frac{n-1}{2}\right]$ take consecutive pairs of integers from $S$, which have difference 1, and apply Theorem~\ref{2 factor Lemma} to generate four $C_n$-factors for each pair. If $|S|$ is odd, there is a leftover difference, $\frac{n-1}{2}$, which is coprime to $n$, so we may use Theorem~\ref{1 factor Lemma} to create two more $C_n$-factors. When $S=\{w\} \ \cup \ \left[w+2,\frac{n-1}{2}\right]$ and $n>2w+3$ we proceed as above, except that the first pair is $(w,w+2)$, which has difference $2$, which is coprime to $n$ since $n$ is odd. When $n=2w+1$ or $2w+3$, then $S=\{w\}$ and our assumptions ensure that $w$ is coprime to $n$. We can then apply Theorem~\ref{1 factor Lemma} to create two  $C_n$-factors on $S$.
\end{proof}

Next we give a negative result which generalises a result of Odabasi and Ozkan in \cite{OO}. This result shows that in general the methods used here will not work for $\alpha=n-1$.

\begin{theorem}
\label{beta=1}
Suppose that $m$ is an odd integer and $n$ is not a multiple of $m$.  Then $(n-1,1)\not\in {\rm HWP}(C_m[n]; m,n)$.
\end{theorem}
\begin{proof}
Suppose that such a factorization existed. Since $m$ is odd, the cycles of the $n-1$ $C_m$-factors must use exactly one point from each part of the $C_m[n]$. 
Thus, after removing the edges contained in the $C_m$-factors, for each point $x_i$, there are two remaining edges incident with $x_i$ which are of the form $((x-1)_j, x_i,)$ and $(x_i,(x+1)_{j'})$. This means that each cycle of the $C_n$-factor cannot reverse direction and must wind around the $m$ parts, i.e.\ it must have the form $(0_{i_1}, 1_{i_2}, 2_{i_3}, \ldots, (m-1)_{i_m}, 0_{i_{m+1}}, 1_{i_{m+2}}, \ldots, (m-1)_{i_{2m}}, \ldots)$, but since $m\nmid n$ this is impossible.
\end{proof}

We now consider the cases when $\alpha=0$ and $\alpha=n$. 
Piotrowski \cite{Piotrowski} has shown the following result for $m \geq 4$. The case $m=3$ is covered by Theorem~\ref{Liu}.
\begin{theorem}
\label{C_m || C_m[n]}
There exists a $C_m$-factorization of $C_m[n]$, i.e. $(n,0) \in {\rm HWP}(C_m[n]; m,n)$, except when $n=2$ and $m$ is odd or $(m,n) =(3,6)$.
\end{theorem}

\begin{theorem}
\label{C_n || C_m[n]}
If $n\geq m\geq 3$ are odd integers then there exists a $C_n$-factorization of $C_m[n]$, i.e.\ $(0,n)\in {\rm HWP}(C_m[n]; m,n)$.
\end{theorem}
\begin{proof}
If $n=m=3$, this is Theorem~\ref{C_m || C_m[n]}. If $n\geq 5$, use the five classes from Theorem~\ref{5 classes} to cover $C_m[\pm\{0,1,2\}]$. 
Use Lemma~\ref{n factor} to obtain a $C_n$-factorization of $C_m\left[\pm\left[3,\frac{n-1}{2}\right]\right]$ and the result follows.
\end{proof} 

\subsection{Cases $\alpha \equiv 0, 5 \pmod{6}$}

We first consider the case when $\alpha \equiv 0 \pmod{6}$. 

\begin{lemma}
If $m$ and $n$ are odd integers with $n\geq m\geq 3$, then $(\alpha, n-\alpha)\in {\rm HWP}(C_m[n]; m, n)$ for every $\alpha\equiv 0\pmod{6}$ with $0\leq \alpha \leq n-5$, except possibly when $(m,n,\alpha) = (3,11,6)$.
\end{lemma}

\begin{proof}
If $\alpha=0$, then the result is Theorem~\ref{C_n || C_m[n]}, so we need only consider the case that $0<\alpha \leq n-5$.

First we suppose that $\alpha >24$.
When $\alpha \equiv 0,6 \pmod{24}$ there exists a Langford sequence of order $\alpha/6$ and defect 3 whenever $\alpha \geq 30$. The triples induced by this sequence, together with their negatives, satisfy Theorem~\ref{corD} and so provide $\alpha$ $C_m$-factors covering $C_m\left[\pm \left[3,\frac{\alpha}{2}+2\right]\right]$. Similarly, when $\alpha \equiv 12,18 \pmod{24}$ we use a hooked Langford sequence of order $\alpha/6$ and defect $3$, which exists whenever $\alpha\geq 30$, to create triples which together with their negatives satisfy Theorem~\ref{corD} and so provide $\alpha$ $C_m$-factors covering $C_m\left[\pm \left[3,\frac{\alpha}{2}+3\right]\setminus \{\pm(\frac{\alpha}{2}+2)\}\right]$.
Take $b=1$ in Theorem~\ref{5 classes}, to get five $C_n$-factors on $C_m[\pm\{0,1,2\}]$.   If $\alpha<n-5$, we now factor 
$C_m\left[\pm \left[\frac{\alpha}{2}+3,\frac{n-1}{2}\right]\right]$ or  
$C_m\left[\left\{\pm \left(\frac{\alpha}{2}+2\right)\right\} \cup \pm \left[\frac{\alpha}{2}+4,\frac{n-1}{2}\right] \right]$ as appropriate into $C_n$-factors by Lemma~\ref{n factor}  with $w=\frac{\alpha}{2}+3$ or $\frac{\alpha}{2}+2$, respectively.

For $\alpha = 6,12,18,24$, the table below gives triples which together with their negatives satisfy Theorem~\ref{corD} and a partition of the remaining integers in $[1,s]$ which satisfy Theorem~\ref{2 factor Lemma}.
The remaining $C_n$-factors on $[s+1,\frac{n-1}{2}]$ can be obtained by Lemma~\ref{n factor}.
\[
\begin{array}{|c|l|c|c|l|} \hline
\alpha & \text{Triples} & n & s & \text{Partition} \\ \hline
6  & \{3,4,6\}  & 13     & 6 & (5) \;\;   \\ \hline
6  & \{3,4,-7\}  & \geq15 & 7 & (5, 6)   \\ \hline
12 & \{4,6,7\}, \{3,5,-8\} & 17     & 8 & \emptyset  \\ \hline
12 & \{4,6,9\}, \{3,5,-8\} & 19     & 9 & (7)  \\ \hline
12 & \{4,6,-10\}, \{3,5,-8\} & \geq21 & 10 & (7, 9)  \\ \hline
18 & \{3,6,-9\}, \{4,7,-11\}, \{5,8,10\} & 23     & 11 & \emptyset  \\ \hline
18 & \{3,10,12\}, \{4,5,-9\}, \{6,8,11\} & 25     & 12 & (7) \\ \hline
18 & \{3,6,-9\}, \{4,7,-11\}, \{5,8,-13\} & \geq27 & 13 & (10, 12)  \\ \hline
24 & \{4,5,-9\}, \{6,11,12\}, \{7,8,14\}, \{3,10,-13\} & 29     & 14 & \emptyset\\ \hline
24 & \{4,5,-9\}, \{6,11,14\}, \{7,8,-15\}, \{3,10,-13\} & 31     & 15 & (12)\\ \hline
24 & \{4,5,-9\}, \{6,11,16\}, \{7,8,-15\}, \{3,10,-13\} & 33 & 16 & (12, 14)\\ \hline
24 & \{4,5,-9\}, \{6,11,-17\}, \{7,8,-15\}, \{3,10,-13\} & \geq33 & 17 & (12, 14, 16)\\ \hline
\end{array}
\]

When $n=11$, so necessarily $\alpha = 6$, and $m>3$, we take $b=1$ in Theorem~\ref{5 classes} to create five $C_n$-factors and create the $C_m$-factors by using matrix of the form $A=\left[
  \begin{array}{r}
     B \\
    -B
  \end{array}
  \right]$
in Theorem~\ref{constructionD}, where $B$ is given below.
\[
  B=\left[
  \begin{array}{rrrrr}
    3 &  4 &  3 & -4 & 5 \\
    4 &  3 & -4 &  5 & 3 \\ 
    5 &  5 &  5 &  3 & 4
  \end{array}
  \right].
\]
\end{proof}

\begin{lemma}
If $n\geq m\geq 3$ are odd integers, then $(\alpha, n-\alpha)\in {\rm HWP}(C_m[n]; m, n)$ for every $\alpha\equiv 5\pmod{6}$ with $0\leq \alpha \leq n$.
\end{lemma}
\begin{proof} 
If $n=\alpha=11$, then the result follows by Theorem~\ref{C_m || C_m[n]}. 
In the other cases, we have that $(\alpha-5, n-\alpha+5)\in HWP(C_m[n]; m,n)$ by the previous Lemma.
The proof of that lemma makes use of 5 $C_n$-factors constructed on $C_m[\pm\{0,1,2\}]$. We can replace them 
with 5 $C_m$-factors we obtain by Theorem~\ref{constructionD} through the matrix $A$ given below:
\[
A=
\left[
  \begin{array}{rrr}
    0  &   1  &  -1  \\
    1  &  -2  &   1  \\ 
   -1  &  -1  &   2  \\
   2   &   0  &  -2  \\
  -2   &   2  &   0  \\
  \end{array}
  \right].
\]
It then follows that $(\alpha, n-\alpha)\in HWP(C_m[n]; m,n)$.
\end{proof}

\subsection{Cases $\alpha \equiv 1, 3 \pmod{6}$}
\
\begin{lemma}
If $m$ and $n$ are odd integers with $n\geq m\geq 3$, then $(\alpha, n-\alpha)\in {\rm HWP}(C_m[n]; m, n)$ for every $\alpha\equiv 1\pmod{6}$ with $0\leq \alpha \leq n$.
\end{lemma}
\begin{proof}
We proceed by defining a set $S \subseteq \mathbb{Z}_n\setminus\{0\}$ which can be covered by a set of triples from a Skolem or Langford sequence, and apply Theorem~\ref{corD} to get $C_m$-factors of $C[S]$. We then define $S^*$ to be the complement of $S$ in $\Z_n\setminus \{0\}$ and partition $C[S^*]$ into $C_n$-factors using Lemmas~\ref{2 factor Lemma} and \ref{1 factor Lemma}. 
We note that the $C_m$-factor, $F$, given below exactly covers $C_m[\{0\}]$, so we only need to find $\alpha-1$ $C_m$-factors.
\[
F = \{(0_i, 1_i, \ldots, (m-2)_i, (m-1)_i)\mid 0\leq i\leq n\}.
\] 
We first deal with the case where $\alpha<n-2$. 

When $u\equiv 0,1\pmod{4}$ we use a Skolem sequence of order $u$, which exists by Theorem~\ref{Skolem}, to partition the set $S= \pm[1, 3u]$ into triples $\{t_1, t_2, t_3\}$ such that $t_1+t_2+t_3=0$.
When $u\equiv 2,3\pmod{4}$ we use a hooked Skolem sequence of order $u$, which exists by Theorem~\ref{Skolem}, to partition the set $S= \pm[1, 3u-1]\cup \pm\{3u+1\}$ into triples $t_1, t_2, t_3$ such that $t_1+t_2+t_3=0$. In each case, there is a $C_m$-factorization of $C_m[S]$ with $\alpha-1$ factors by Theorem~\ref{corD}.
The complement of $S$ in $\Z_n\setminus \{0\}$ satisfies the conditions of Lemma~\ref{n factor} and so can be factored into $C_n$-factors.

We now consider the case when $\alpha=n-2$. Note that in this case, $n = 6u+3 \equiv 3 \pmod{6}$.

If $u\equiv 0,1\pmod{4}$, use $S= \pm[1, 3u]$ as above and note that $3u+1= \frac{n-1}{2}$.  We apply Lemma~\ref{1 factor Lemma} to create two $C_n$-factors using this difference. 

If $u\equiv 3\pmod{4}$ we instead partition $S=\pm[2,3u+1]$. Note that in this case $u \geq 3$, so a Langford sequence of order $u$ and defect $2$ exists by Theorem~\ref{lang} and we may use Theorem~\ref{corD} to create the $C_m$-factors. The remaining difference is 1 and we apply Lemma~\ref{1 factor Lemma} to create 2 $C_n$-factors using this difference.

If $u\equiv 2\pmod{4}$ we instead partition $S=\pm[2, 3u]\cup \pm\{3u+2\}$ using a hooked Langford sequence of order $u$ and defect $2$, which exists when $u\geq 2$ by Theorem~\ref{lang}, and apply Theorem~\ref{corD} to create the $C_m$-factors. Note that $3u+2=\frac{n+1}{2} \equiv -\left(\frac{n-1}{2}\right) \equiv -(3u+1)\pmod{n}$ and so $S= \pm [2,\frac{n-1}{2}]$. 
The remaining difference is 1 and we apply Lemma~\ref{1 factor Lemma} to create two $C_n$-factors using this difference.
\end{proof}

\begin{lemma}
If $m$ and $n$ are odd integers with $n\geq m\geq 3$, then $(\alpha, n-\alpha)\in {\rm HWP}(C_m[n]; m, n)$ for every $\alpha\equiv 3\pmod{6}$ with $0\leq \alpha \leq n$.
\end{lemma}
\begin{proof}
We write $\alpha = 6u+3$ and note that if $\alpha=n$ this is Lemma~\ref{C_m || C_m[n]}, so we may assume $0<\alpha<n$.
When $u\equiv 0,1\pmod{4}$ we use a Skolem sequence of order $u$, which exists by Theorem~\ref{Skolem}, to partition the set $S= \pm[1, 3u]$ into triples $\{t_1, t_2, t_3\}$ such that $t_1+t_2+t_3=0$.
When $u\equiv 2,3\pmod{4}$ we use a hooked Skolem sequence of order $u$, which exists by Theorem~\ref{Skolem}, to partition the set $S= \pm[1, 3u-1]\cup \pm\{3u+1\}$ into triples $\{t_1, t_2, t_3\}$ such that $t_1+t_2+t_3=0$. 
In each case, there is a $C_m$-factorization of $C_m[S]$ with $\alpha-3$ factors by Theorem~\ref{corD}.
Let
\[d = \left\{ 
\begin{array}{cl}
3u+1 & \mbox{ when } u \equiv 0,1 \pmod{4}\\
3u & \mbox{ when } u\equiv 2,3 \pmod{4} \\
\end{array}
\right.
\mbox{and }
A = \left[
\begin{array}{ccc}
d & -d & 0 \\
0 & d & -d \\
-d & 0 & d \\
\end{array}
\right].
\]
$A$ satisfies all of the properties of Theorem~\ref{constructionD}, and so gives a $C_m$-factorization of $C_m[\pm\{0,d\}]$ into three $C_m$-factors.

We use the remaining differences, $S^* = \pm\left[3u+2, \frac{n-1}{2}\right]$, to create $C_n$-factors by Lemma~\ref{n factor}.
\end{proof}

\subsection{Case $\alpha \equiv 2 \pmod{6}$}

Let $n\geq m\geq 3$ be odd.
Given positive integers $a$, $b$, and $u$, we define $U_{a,b} = \{a, 2a, 3a, 4a \} \cup \{b, 2b\}$ and
$S_{u,a,b} = [1,u]\setminus U_{a,b}$. Also, 
we define the matrix $A_a=\left[
  \begin{array}{r}
     B_a \\
    -B_a
  \end{array}
  \right],$
where $B_a$ is defined as  
\[
  B_{a}= a\left[
  \begin{array}{rrr}
    1 &   3 & -4\\
    2 &    1 & -3 \\
    3 &  -4 &  1  \\
    4 &  -2  & -2
  \end{array}
  \right].
\]

\begin{lemma}
\label{gen 2}
Let $n \geq m \geq 3$ be odd integers, and let $\alpha \leq n-5$ be a positive integer with $\alpha \equiv 2 \pmod{6}$.  Suppose that there exist positive integers $a$ and $b$, with $b$ coprime to $n$, so that 
\begin{enumerate}
\item $U_{a,b} \subseteq [1,(\alpha+4)/2]$, and 
\item there exist triples $\pm T_i$, $1 \leq i \leq (\alpha-8)/6$, satisfying the conditions of Theorem~\ref{corD} on the set $\pm S_{(\alpha+4)/2,a,b}$,
\end{enumerate}
Then $(\alpha,n-\alpha)\in {\rm HWP}(C_m[n]; m,n)$.
\end{lemma} 

\begin{proof}
Let $u = (\alpha+4)/2$.
Note that the elements of $\pm U_{a,b}$ are all distinct in $\Z_n$ by condition $1$. Hence,
$A_a$ satisfies the assumptions of Theorem~\ref{constructionD}, and so there are eight $C_m$-factors covering $C_m[\{\pm a,$ $\pm 2a,\pm 3a,\pm 4a\}]$.
Furthermore, $b$ is coprime to $n$, so $b$ satisfies the conditions of Theorem~\ref{5 classes} and there are five $C_n$-factors covering $C_m[\{0, \pm b, \pm 2b\}]$. 
In view of condition $2$, there are $\alpha-8$ $C_m$-factors covering $C_m[S_{u,a,b}]$. 
Between them this exactly covers $C_m[S]$, where $S = \pm [0,u]$. The remaining differences are $S^* = \pm \left[u+1, \frac{n-1}{2}\right]$ and so $C_m[S^*]$ can be factored into $C_n$-factors by Lemma~\ref{n factor}.
\end{proof}

\begin{lemma}
\label{alpha=2}
If $n$ and $m$ are odd integers with $n\geq m\geq 3$, $\alpha\equiv 2\pmod{6}$, $\alpha \geq 8$ and $n\geq\alpha+5$, then $(\alpha, n-\alpha)\in {\rm HWP}(C_m[n],m,n)$, except possibly when $(m,n,\alpha)=(3,13,8),(3,15,8)$.
\end{lemma}
\begin{proof}
We first assume that $\alpha \geq 224$ and set $u = (\alpha+4)/2\geq 114$. We consider the cases $n=\alpha+5$ and  $n\geq \alpha+7$ separately.

If $n=\alpha+5$, we take $a=3$, $b=4$ in Lemma~\ref{gen 2}. 
We now obtain the required triples on $\pm S_{u,3,4} = \pm [13,u] \cup \pm \{1,2,5,7,10,11\}$. First take $T_1 = \{2,5,-7\}$ and $T_2 = \{1,10,-11\}$.  
Now, let $\nu = (u-12)/3\geq 34$, and take a (hooked) Langford sequence of order $\nu$ and defect $13$, which exists by Theorem~\ref{lang}, to partition the differences $\pm [13, u]$ into triples satisfying the conditions of Theorem~\ref{corD} and we are done. 

Now, let  $n\geq \alpha+7$ and set $\nu = (u-15)/3\geq 33$.
Using $b=8$ in Theorem~\ref{5 classes} gives five $C_n$-factors on $C_m[\pm \{0,8,16\}]$.  Moreover, $A_3$ satisfies the conditions of Theorem~\ref{constructionD} to give eight $C_m$-factors on $C_m[\pm \{3,6,9,12\}]$. By Theorem~\ref{corD}, we obtain the remaining $C_m$-factors by partitioning the set $S \cup \pm \{1,4,5,7,10,11,13,14,15\}$ into triples where either $S=\pm [17,u+1]$ or $\pm [17,u+2]\setminus\{u+1\}$ according to whether $\nu\equiv 0,1\pmod{4}$ or $\nu\equiv 2,3\pmod{4}$. First take $T_1 = \{1,13,-14\}$, $T_2 = \{4,7,-11\}$ and $T_3=\{5,10,15\}$. 
It remains to find triples which partition $S$.

If $\nu\equiv 0,1\pmod{4}$
we take a Langford sequence of order $\nu$ and defect $17$, which exists by Theorem~\ref{lang}, to partition the differences $S = \pm [17, u+1]$ into triples satisfying the conditions of Theorem~\ref{corD}. In this case 
we create two $C_n$-factors on differences $\pm 2$ using Theorem~\ref{1 factor Lemma} and create the remaining $n-\alpha-7$ $C_n$-factors by applying Lemma~\ref{n factor} to the remaining differences $\pm[u+2, (n-1)/2]$.

If $\nu\equiv 2,3\pmod{4}$ we take a hooked Langford sequence of order $\nu$ and defect $17$, which exists by Theorem~\ref{lang}, to partition the differences $S = \pm [17, u+2]\setminus \{\pm (u+1)\}$ into triples satisfying the conditions of Theorem~\ref{corD}. 
In this case, if $n=\alpha+7$, to create the remaining $2$ $C_n$-factors 
we may apply Lemma~\ref{1 factor Lemma} to the differences $\pm 2$.
If $n>\alpha+7$, we create the $C_n$-factors by applying Lemma~\ref{2 factor Lemma} to the pair $(2, u+1)$ and applying
Lemma~\ref{n factor} to the remaining differences  $\pm[u+3, (n-1)/2]$.

Next we consider the case where $\alpha = 8$. First, we use $b=1$ in Theorem~\ref{5 classes} to create five $C_n$-factors on $C_m[0,\pm 1,\pm 2]$.
When $n\geq 17$, $n\neq 21$, $A_3$ satisfies the conditions of Theorem~\ref{constructionD}. 
The absolute value of the elements of $A_3$ and the remaining differences are given in the table below. This table also shows how to partition the remaining differences into pairs whose difference is coprime to $n$ and possibly a leftover value that is coprime to $n$ and so Theorems~\ref{2 factor Lemma} and \ref{1 factor Lemma} apply.
\[
\begin{array}{|c|c|l|}\hline
n & \text{Elements of } A_3 & \text{Remaining Differences} \\ \hline
17 & 3,6,8,5 & (4,7) \\ \hline
19 & 3,6,9,7 & (4),(5,8) \\ \hline
23 & 3,6,9,11 & (4),(5,7),(8,10) \\ \hline
\geq 25 & 3,6,9,12 & (4,5),(7,8),(10,11), (13,14), \ldots, \left(\frac{n-3}{2}, \frac{n-1}{2}\right) \text{ or }\\ 
& & (4,5),(7,8),(10,11), (13,14), \ldots, \left(\frac{n-5}{2},\frac{n-3}{2}\right), \left(\frac{n-1}{2}\right) \\ 
\hline
\end{array}
\]
When $n=21$, $A_4$ satisfies the conditions of Theorem~\ref{constructionD}, and the remaining differences can be paired as $(3,7)$ and $(6,10)$.
When $m>3$, and $n=13$ or $15$ we use the matrices of the form $A=\left[
  \begin{array}{r}
     B \\
    -B
  \end{array}
  \right]$
in Theorem~\ref{constructionD}, where $B$ is given below.
\[
\begin{array}{cc}
\begin{array}{rcc}
B & = & \left[
  \begin{array}{rrrrr}
    3 & -5 & 4  & -6  & 4 \\
    4 & -3 & -5 &  -3 & -6\\
    5 & 6  & 6  & 4 & 5    \\
    6 & -4 & 3  & 5 & 3  
  \end{array}
  \right], \\
  & & n = 13 \\
\end{array}
& 
\begin{array}{rcc}
B & = & \left[
  \begin{array}{rrrrr}
    3 & -7 &  7 & -6 & 3 \\
    5 &  3 & -6 &  7 & 6 \\
    6 & -5 & -3 & -5 & 7 \\
    7 &  6 & -5 & -3 & -5 
  \end{array}
  \right].\\
	& & n = 15 \\
\end{array} 
\end{array} 
\]
If $n=13$, we are done.  When $n=15$, the remaining differences are $\pm 4$ which are coprime to $n$, so we can apply Theorem~\ref{1 factor Lemma} to form two $C_n$-factors on $C_m[\pm 4]$.

When $\alpha=14$, $n\geq 25$, using $b=1$ in Theorem~\ref{5 classes} gives five $C_n$-factors on $C_m[\pm \{0,1,2\}]$.  Moreover, $A_3$ satisfies the conditions of Theorem~\ref{constructionD} to give eight $C_m$-factors on $C_m[\pm \{3,6,9,12\}]$, and the triple $T_1=\{4,10,-14\}$ can be used in Theorem~\ref{corD} to give six further $C_m$-factors.  The following table shows how to partition the remaining differences in $[1,\ldots,\frac{n-1}{2}]$ into pairs with difference relatively prime to $n$ and singletons relatively prime to $n$; applying Lemmas~\ref{2 factor Lemma} and \ref{1 factor Lemma} give the remaining $C_n$-factors. When $n \geq 31$, the differences in $[16,\frac{n-1}{2}]$ can be partitioned by Lemma~\ref{n factor}.
\[
\begin{array}{|c|l|l|} \hline
n & \mbox{Elements of $A_4$ and $T_1$} & \mbox{Remaining differences} \\ \hline
25 & 3,6,9,12,  4,10,11     & (5,7), (8) \\
27 & 3,6,9,12,  4,10,13     & (5)(8),(7,11) \\
29 & 3,6,9,12,  4,10,14     & (5,7), (8,11), (13) \\
\geq 31 & 3,6,9,12,  4,10,14     & (5,7), (8,11), (13,15), [16,\frac{n-1}{2}] \\ \hline
\end{array}
\]

 For $n=19,21$ and $23$,  we use the matrices of the form $A=\left[
  \begin{array}{r}
     B \\
    -B
  \end{array}
  \right]$
in Theorem~\ref{constructionD}, where $B$ is given below.
\[
\begin{array}{l}
B = \left[
  \begin{array}{rrr}
    3 &  6 & -9  \\
    4 &  4 & -8  \\ 
    5 & -8 &  3  \\
    6 &  9 &  4  \\
    7 &  7 &  5  \\
    8 &  5 &  6  \\
    9 &  3 &  7
  \end{array}
  \right], \\
  \hspace{1.5cm} n = 19 \\
\end{array}
\begin{array}{l}
B = \left[
\begin{array}{rrr}
 3 & 3  & -6 \\
 4 & -7 & 3 \\
 5 & 8  & 8 \\
 6 & 6  & 9 \\
 7 & 9  & 5 \\
 8 & -4 & -4 \\
 9 & 5  & 7 \\
\end{array}
\right],\\
  \hspace{1.5cm} n = 21
\end{array}
\begin{array}{l}
B  = \left[
\begin{array}{rrr}
 -4 & 11  & -7 \\
 8 & 4 & 11 \\
 11 & -7  & -4 \\
 7 & 8  & 8 \\ 
 3 & 6  & -9 \\
 6 & -9 & 3 \\
 -9 & 3  & 6 \\
\end{array}
\right]. \\
  \hspace{1.5cm} n = 23
\end{array}
\]
Take $b=1$ in Theorem~\ref{5 classes} to form five $C_n$-factors on $C_m[\pm\{0,1,2\}]$.  If $n=21$, two further $C_n$-factors arise by applying Theorem~\ref{1 factor Lemma} on differences $\pm 10$. If $n=23$, four further $C_n$-factors arise by applying Theorem~\ref{1 factor Lemma} on differences $\pm 5$ and $\pm 10$.

When $\alpha = 20$, take $a = 3$, $b = 4$, in Lemma~\ref{gen 2} together with the triples $T_1=\{2,5,-7\}$, $T_2=\{1, 10, -11\}$. 

When $\alpha = 26$, $n > 31$, form the $C_m$-factors by using $A_3$ in Theorem~\ref{constructionD} together with the triples $T_1=\{ 1, 13, -14 \}$, $T_2=\{ 4, 7, -11 \}$, $T_3=\{ 5, 10, -15 \}$ in Theorem~\ref{corD}.  Taking $b=8$ in Theorem~\ref{5 classes} gives five $C_n$-factors and the differences $\pm 2$ give rise to two more $C_n$-factors by Theorem~\ref{1 factor Lemma}. The remaining differences $\pm [17,\ldots,\frac{n-1}{2}]$ can be used to form $C_n$-factors by Lemma~\ref{n factor}.
When $n = 31$, take $a=4$ and $b=7$ along with the triples $T_1=\{1,5,-6\}$, $T_2=\{2,9,-11\}$ and $T_3=\{3,10,-13\}$ in Lemma~\ref{gen 2}.  

For the cases $26 < \alpha < 224$, Appendix~\ref{2 mod 6 Appendix} gives values of $a$ and $b$ and $(\alpha-8)/6$ triples which satisfy the conditions of Lemma~\ref{gen 2}.
\end{proof}

\subsection{Case $\alpha \equiv 4 \pmod{6}$}

We proceed similarly to the case $\alpha \equiv 2\pmod{6}$. Let $n\geq m\geq 3$ be odd.
Given positive integers $a$, $b$ and $u$, we define $U_{a,b}' = \{0,a,2a,4a,5a,6a,b,2b\}$ and
$S'_{u,a,b} = [0,u]\setminus U_{a,b}'$.
Also, we define the matrix $A'_a=\left[
  \begin{array}{r}
     B'_a \\
    -B'_a
  \end{array}
  \right]$,
where $B'_a$ is defined as  
\[
  B'_{a}= a\left[
  \begin{array}{rrr}
    1 & -5 &  4 \\
    2 &  4 & -6 \\
   -4 &  2 &  2 \\
    5 & -6 &  1 \\
   -6 &  1 &  5
  \end{array}
  \right].
\]

\begin{lemma}
\label{gen 4}
Let $n \geq m \geq 3$ be odd integers and let $\alpha\leq n-5$ be a positive integer with $\alpha \equiv 4 \pmod{6}$. Suppose that there exist positive integers $a$ and $b$, with $b$ coprime to $n$, so that
\begin{enumerate}
\item $U'_{a,b} \subseteq [1,(\alpha+4)/2]$, and 
\item there exist triples $\pm T_i$, $1 \leq i \leq (\alpha-10)/6$, satisfying the conditions of Theorem~\ref{corD} on the set $\pm S'_{(\alpha+4)/2,a,b}$,
\end{enumerate}
then $(\alpha,n-\alpha)\in {\rm HWP}(C_m[n]; m,n)$.
\end{lemma} 

\begin{proof} 
The proof is similar to the proof of Lemma~\ref{gen 2}, except that we use $A'_a$, $U'_a$ and $S'_{u,a,b}$ in place of $A_a$, $U_a$ and $S_{u,a,b}$.
\comment{
\pnote{The following is to make sure that the claim above is true and should be removed in the final version\\}
Let $u = (\alpha+4)/2$.
Note that if the elements of $\pm U_{a,b}$ are all distinct, then $A_a$ satisfies the conditions of Theorem~\ref{constructionD}, and so there are ten $C_m$-factors covering $C_m[\{\pm a,\pm 2a,\pm 4a, \pm 5a, \pm 6a\}]$.
Furthermore, $b$ is coprime to $n$, so $b$ satisfies the conditions of Theorem~\ref{5 classes} and there are five $C_n$-factors covering $C_m[\{0, \pm b, \pm 2b\}]$. The triples $\pm T_i$, $1 \leq i \leq (\alpha-10)/6$, satisfy the conditions of Theorem~\ref{corD}, and so there are $\alpha-10$ $C_m$-factors covering $C_m[S_{u,a,b}]$. Between them this exactly covers $C_m[S]$, where $S = \pm [0,u]$.

We use the remaining differences to create the remaining $C_n$-factors. 
Let $S^*$ be the complement of $\pm [0,u]$ in $\Z_n\cap \left[0,\frac{n-1}{2}\right]$. 
We pair up consecutive differences from $\left[u+1,\frac{n-1}{2}\right]$ and apply Lemma~\ref{2 factor Lemma} to these pairs of differences to create 4 $C_n$-factors for each pair. If $\beta = n-\alpha \equiv 3\pmod{4}$ one difference remains, namely $\frac{n-1}{2}$, which is coprime to $n$, we use this difference in Lemma~\ref{1 factor Lemma} to create 2 more $C_n$-factors.
}\end{proof}

\begin{lemma}
\label{alpha=4}
If $m$ and $n$ are odd integers with $n\geq m\geq 3$, $\alpha\equiv 4\pmod{6}$, $\alpha > 4$ and $n\geq\alpha+5$, then $(\alpha, n-\alpha)\in {\rm HWP}(C_m[n],m,n)$.
\end{lemma}
\begin{proof}
We first deal with the case where $\alpha \geq 268$. Let $u =(\alpha+4)/2$ and $\nu = (u-19)/3$ and take $a= 3$ and $b=8$ in Lemma~\ref{gen 4}, we now obtain the required triples on $S'_{u,3,8}$.
First take
$T_1 = \{1,10,-11\}$, $T_2 = \{4,13,-17\}$, $T_3=\{2,7,-9\}$, $T_4 = \{5,14,-19\}$.  It remains to find triples which partition $[20, u]$.

If $\nu\equiv 0,3 \pmod{4}$, we take a Langford sequence of order $\nu$ and defect $20$, which exists whenever $\nu\geq 39$ ($\alpha \geq 268$), to partition the differences $S = \pm [20, u]$ into triples satisfying the conditions of Theorem~\ref{corD} and we are done. Otherwise, if $\nu\equiv 1, 2\pmod{4}$ we take a hooked Langford sequence of order $\nu$ and defect $20$, which exists whenever $\nu\geq 39$ ($\alpha \geq 268$), to partition the differences $S = \pm [20, u+1]\setminus \{\pm u\}$ into triples satisfying the conditions of Theorem~\ref{corD}. 
We note that when $n = 2u+3$, so $n=\alpha+7$, since $\alpha\equiv 4\pmod{6}$,  $u \equiv 1\pmod{3}$, and so we can apply Lemma~\ref{n factor} to $S$ to obtain the $C_n$-factors.

For the cases $28 < \alpha < 268$, Appendix~\ref{4 mod 6 Appendix} gives values of $a$ and $b$ and $(\alpha-10)/6$ triples which satisfy the conditions of Lemma~\ref{gen 4}.
For the cases $10 \leq \alpha \leq 28$, in most cases we proceed similarly to Lemma~\ref{gen 4} and find values for $a$ and $b$, with $b$ coprime to $n$, so that the elements of $U'_{a,b}$ are distinct, and $(\alpha-10)/6$  triples satisfying the conditions of Theorem~\ref{corD} on a set disjoint from $U'_{a,b}$. It then remains to show that the remaining differences can be partitioned into pairs and singletons satisfying Theorems~\ref{2 factor Lemma} and \ref{1 factor Lemma} respectively. 

When $10 \leq \alpha \leq 28$ and $n > 37$, we take $a=3$, $b=8$. The differences from $[20,\frac{n-1}{2}]$ satisfy the conditions of Lemma~\ref{n factor} and so can be used to create $C_n$-factors. The required triples and a partition of the leftover differences in $[1,19]$ into pairs whose difference is coprime to $n$ and singletons coprime to $n$ are given in the table below.
\[
\begin{array}{|c|l|l|} \hline
\alpha & \text{Triples} & \text{Leftover Differences} \\ \hline
10 &\multicolumn{1}{c|}{-} & (1,2), (4,5), (7,9), (10,11), (13,14), (17,19)
\\ \hline
16 & \{4, 13, -17\} &  (2), (1, 5), (7,9), (10, 14), (11, 19)
\\ \hline
22 & \{4, 13, -17\}, \{5,  14, -19\} &  (1,2), (7, 9), (10, 11) \\ \hline
28 & \{4, 13, -17\}, \{5, 14, -19\}, & (1), (10, 11)  \\
   &  \{2, 7, -9\} & \\ \hline
\end{array}
\]

We now consider the cases where $10 \leq \alpha \leq 28$ and $\alpha+5 \leq n \leq 37$.  
First, when $\alpha=10$ and $n=15$, $17$ or $21$, then the following matrices can be used to form ten $C_m$-factors by Theorem~\ref{constructionD}.  
\[
\begin{array}{ccccc}
\left[
\begin{array}{rrr}
 3 & 5 & 7 \\
 4 & 6 & 5 \\
 5 & 7 & 3 \\
 6 & 3 & 6 \\
 7 & 4 & 4 \\
\end{array}
\right]
& \hspace*{0.5cm} & 	
\left[
\begin{array}{rrr}
 3 & 3 & -6 \\
 4 & 4 & -8 \\
 5 & 8 & 4 \\
 6 & 6 & 5 \\
 8 & -5 & -3 \\
\end{array}
\right]
& \hspace*{0.5cm} & 
\left[
\begin{array}{rrr}
 3 & 5 & -8 \\
 4 & -8 & 4 \\
 5 & 4 & -9 \\
 8 & -3 & -5 \\
 9 & 9 & 3 \\
\end{array}
\right].
 \\
n=15 & & n=17 & & n=21
\end{array}
\]
Taking $b=1$ in Theorem~\ref{5 classes} gives five $C_n$-factors on $C_m[\pm \{0,1,2\}]$.  If $n=15$, we are done.  If $n=17$, the remaining difference is 7, and Lemma~\ref{1 factor Lemma} gives two $C_n$-factors on $C_m[\pm\{7\}]$.  If $n=21$, the remaining differences are 6, 7 and 10; Lemma~\ref{2 factor Lemma} yields four $C_n$-factors on $C_m[\pm\{6,7\}]$ and Lemma~\ref{1 factor Lemma} give two $C_n$-factors on $C_m[\pm \{10\}]$.  

If $\alpha=10$ and $n=19$ or $23 \leq n \leq 37$, take $a=1$ and $b=8$ to form ten $C_m$-factors and six $C_n$-factors with differences in $\pm U'_{a,b}$ as in the proof of Lemma~\ref{gen 4}. 
The table below gives a partition of the remaining differences into pairs with difference coprime to $n$ and singletons coprime to $n$.  
\[
\begin{array}{|c|l|} 
\multicolumn{2}{l}{\alpha=10} \\ \hline
n &  \text{Leftover Differences} \\ \hline
19 & (7,9) \\ \hline
23 & (3,9), (10, 11) \\ \hline
25 & (3,7), (10,11), (12) \\ \hline
27 & (3,7), (9, 10), (12, 13) \\ \hline
29 & (3,7), (9, 10), (11), (12, 14) \\ \hline
31 & (3,7), (9, 10), (11, 12), (13, 14) \\ \hline
33 & (3,7), (9, 11), (10), (12, 13), (14, 15) \\ \hline
35 & (3,7), (9, 10), (11, 12), (13, 14), (15,17) \\ \hline
37 & (3,7), (9,10), (11,12), (13,14), (15,17), (18) \\ \hline
\end{array} 
\]

When $\alpha=16$ and $n=21$, we use a matrix of the form $A=\left[
  \begin{array}{r}
     B \\
    -B
  \end{array}
  \right]$
in Theorem~\ref{constructionD}, where 
\[
  B = \left[
  \begin{array}{ccc}
	3 & 4 & -7 \\
	4 & 8 &  9 \\
	5 & 6 & 10 \\
	6 & 7 &  8 \\
	7 & 10 & 4 \\
	8 & -3 & -5 \\
	9 &  9 &  3 \\
	10 & 5 &  6 \\
  \end{array}
  \right].
\]
For $\alpha=16$ and $23 \leq n \leq 37$, take $a=1$ and $b = 8$ to factor $C_m[\pm U'_{a,b}]$ into ten $C_m$-factors and five $C_n$-factors as in the proof of Lemma~\ref{gen 4}. The table below gives $(\alpha-10)/6=1$ triple of differences which is used to form six further $C_m$-factors, and a partition of the leftover differences to form $C_n$-factors for each value of $n$.
\[
\begin{array}{|c|c|l|}
\multicolumn{3}{l}{\alpha=16} \\ \hline
n & \text{Triples} & \text{Leftover Differences} \\ \hline
23 & \{3, 9, 11\} & (10) \\ \hline
25 & \{3, 7, -10\} & (11,12) \\ \hline
27 & \{3, 9, -12\} &  (7), (10), (13) \\ \hline
29 & \{3, 9, -12\} & (7, 11), (10, 14) \\ \hline
31 & \{3, 7, -10\} & (9, 11), (12, 13), (14) \\ \hline
33 & \{3, 7, -10\} & (9, 11), (12,13), (14, 15) \\ \hline
35 & \{3, 7, -10\} & (9,11) (12,13) (14, 15), (17)\\ \hline
37 & \{3, 7, -10\} &  (9,11) (12,13) (14,15) (17,18) \\ \hline
\end{array}
\]

When $\alpha=22$ and $n=27$, we construct five $C_n$-factors by taking $b=8$ in Theorem~\ref{5 classes}. We construct twelve $C_m$-factors by taking the triples $T_1=\{6,7,-13\}, T_2=\{5,10,12\}$ in Theorem~\ref{corD} and a further ten $C_m$-factors using a matrix of the form $A=\left[
  \begin{array}{r}
     B \\
    -B
  \end{array}
  \right]$
in Theorem~\ref{constructionD}, where 
\[
  B = \left[
  \begin{array}{ccc}
    1 &   3 & -4\\
    2 &    1 & -3 \\
    3 &  -4 &  1  \\
    4 &  -2  & -2 \\
    9 & 9 & 9 \\
  \end{array}
  \right].
\]

For $\alpha=22$ and $29 \leq n \leq 37$, take $a=1$ and $b=8$ factor $C_m[\pm U'_{a,b}]$ into ten $C_m$-factors and six $C_n$-factors
as in the proof of Lemma~\ref{gen 4}. The table below gives $(\alpha-10)/6=2$ triples of differences which are used to form twelve further $C_m$-factors, and a partition of the leftover differences to form $C_n$-factors.
\[
\begin{array}{|c|l|l|}
\multicolumn{3}{l}{\alpha=22} \\ \hline
n  & \text{Triples} & \text{Leftover Differences} \\ \hline
29 & \{7,10,12\}, \{3,11,-14\} & (9) \\ \hline
31 & \{3, 9, -12\}, \{7,  10, 14\} & (11, 13)  \\ \hline
33 & \{3, 12, -15\}, \{9,  11, 13\} & (7), (10, 14)  \\ \hline
35 & \{3, 7, -10\}, \{9,  12, 14\} & (11, 13), (15, 17) \\ \hline
37 &  \{3, 7, -10\}, \{11,  12, 14\} & (9, 13) (15, 17), (18) \\ \hline
\end{array}
\]

Finally, for $\alpha=28$ and $33 \leq n \leq 37$, the following table gives the values of $a$ and $b$ to factor $C_m[\pm U'_{a,b}]$ as in Lemma~\ref{gen 4}, $(\alpha-10)/6=3$ triples of differences which are used to form 18 further $C_m$-factors, and a partition of the leftover differences to form further $C_n$-factors.
\[
\begin{array}{|c|c|c|l|l|}
\multicolumn{5}{l}{\alpha=28} \\ \hline
n & a & b & \text{Triples} & \text{Leftover Differences} \\ \hline
33 & 2 & 13 & \{1,14,-15\}, \{5,11,-16\}, \{3,6,-9\}  
&\multicolumn{1}{c|}{-} \\ \hline
35 & 2 & 7 & \{1, 5,-6\}, \{3, 13, -16\}, \{9, 11, 15\} & (17) \\ \hline
37 & 1 & 8 & \{3, 7, -10\}, \{11,  12, 14\}, \{9, 13, 15\} & (17, 18)  \\ \hline
\end{array}
\]
\end{proof}

The results of this section are summarized in Theorem~\ref{C_m[n]}, which we restate below.
\newtheorem*{thm:C_m[n]}{Theorem \ref{C_m[n]}}
\begin{thm:C_m[n]}
If $n$, $m$ and $\alpha$ are odd integers with $n\geq m\geq 3$, $0\leq \alpha \leq n$, then $(\alpha, \beta)\in {\rm HWP}(C_m[n]; m, n)$, if and only if $\beta=n-\alpha$, except possibly when $\alpha= 2,4$, $\beta = 1, 3$, or $(m,n,\alpha) = (3,11,6)$, $(3,13,8)$, $(3,15,8)$.
\end{thm:C_m[n]}

\section{Main Theorem}
\label{Section Main}

In this section we prove the main result. We first prove a weaker result in the case when $v=mn$ and then prove the main result for $v>mn$.

\begin{theorem}
\label{HW t=1}
If $m$ and $n$ are odd integers with $n \geq m \geq 3$, then $(\alpha,\beta) \in \mathrm{HWP}(mn; m, n)$ if and only if $\alpha+\beta=(mn-1)/2$, except possibly when $\beta \in [1,(n-3)/2] \cup \{(n+1)/2,(n+5)/2\}$, $(m,\alpha) = (3,2)$, $(3,4)$, or $(m,n,\alpha,\beta) = (3,11,6,10)$, $(3,13,8,10)$, $(5,7,9,8)$, $(5,9,11,11)$, $(5,9,13,9),$
$(7,9,20,11)$, $(7,9,22,9)$.
\end{theorem}
\begin{proof}
We first note that the condition that $\alpha+\beta = (mn-1)/2$ is necessary by Theorem \ref{nec}. When $m=n$, this is equivalent to the uniform Oberwolfach problem and is covered by Theorem~\ref{OP uniform}; we thus assume that $n>m$. The case $(m,n) = (3,15)$ is solved in \cite{AdaBilBryElz}, except when $\beta=1$, so we may assume $(m,n) \neq (3,15)$.

By Theorem~\ref{Hamiltonian}, there is a $C_m$-factorization of $K_m$  with $r=(m-1)/2$ $C_m$-factors. Expand each point of this factorization by $n$ to get a $C_m[n]$-factorization of $K_m[n]$ with $r$ $C_m[n]$-factors. 
This design has $m$ uncovered holes of size $n$; fill each of these holes of size $n$ with a Hamiltonian factorization of $K_n$, to get $(n-1)/2$ $C_n$-factors. 

Set $\alpha = xn+y$, with $0 \leq x < r$ and $0 \leq y < n$; hence, $\beta = (n-1)/2 + (r-x)n-y$. 
If the conditions of Theorem~\ref{C_m[n]} are satisfied for $y$, i.e.\ there is an HW$(C_m[n]; m, n; y,n-y)$, then we fill $x$ resolution classes with an HW$(C_m[n]; m, n;  n, 0)$, one class with an HW$(C_m[n]; m, n; y,n-y)$.  
The remaining $r-x-1$ classes are filled with an $\mathrm{HW}(C_m[n];m,n;0,n)$.

Now suppose that the conditions of Theorem~\ref{C_m[n]} are not satisfied for $y$, i.e.\ $y \in \{2,4,n-3,n-1\}$ or $(m,n,y) = (3,11,6)$, $(3,13,8)$ (we are supposing $(m,n)\neq (3,15)$). 
If $x \leq r-2$, then we fill $x$ of the $C_m[n]$-factors with an $\mathrm{HW}(C_m[n];m,n;n,0)$, one with an $\mathrm{HW}(C_m[n];m,n;1,n-1)$, one with an $\mathrm{HW}(C_m[n];m,n;y-1,n-y+1)$ and the remaining $r-x-2$ classes with an $\mathrm{HW}(C_m[n];m,n;0,n)$.  

We are left with the case $x=r-1$. 
Note that if $m=3$, then $r=1$, so the remaining factors can be found if and only if $(y,n-y) \in \mathrm{HWP}(C_m[n];m,n)$, leading to the possible exceptions when $(m,n,\alpha)=(3,11,6)$ or $(3,13,8)$, or $(m,\alpha) = (3,2), (3,4)$. 
Now suppose that $m \geq 5$, so that $r \geq 2$, but $n>m$, hence $n\geq7$.  
Note that since $\beta \neq (n+1)/2, (n+5)/2$, we have that $y \neq n-1, n-3$. 

We are thus left with the case $y\in \{2,4\}$ and $y \leq n-5$. If $n \geq 11$, we fill $(x-1)$ $C_m[n]$-factors with an $\mathrm{HW}(C_m[n];m,n;n,0)$, one with an  $\mathrm{HW}(C_m[n];m,n;n-5,5)$ (Theorem~\ref{C_m[n]}, $n-5>4$) and one with an $\mathrm{HW}(C_m[n];m,n;y+5,n-y-5)$ (Theorem~\ref{C_m[n]}, $y+5=7,9$). 
For $n=7,9$ we obtain the possible exceptions when $(m,n,\alpha)\in \{(5,7,9,8), (5,9,11,11)$, $(5,9,13,9),$
$(7,9,20,11)$, $(7,9,22,9)\}$.
\end{proof}

We are now ready to prove the main theorem, Theorem~\ref{main}, in a similar manner to the proof above for $v=mn$.

\begin{theorem} \label{HW t>1}
If $m$ and $n$ are odd integers with $n \geq m \geq 3$ and $t>1$, then $(\alpha, \beta)\in {\rm HWP}(mnt; m, n)$ if and only if $t$ is odd, $\alpha, \beta \geq 0$ and $\alpha+\beta = (mnt-1)/2$, except possibly when $\beta = 1$ or $3$, or $(m,n,\beta) = (5,9,5)$, $(5,9,7)$, $(7,9,5)$, $(7,9,7)$, $(3,13,5)$.
\end{theorem}
\begin{proof}
We first note that the conditions that $t$ be odd and $\alpha+\beta = (mnt-1)/2$ are necessary by Theorem \ref{nec}. We now show sufficiency. When $m=n$ (or equivalently $\alpha=0$ or $\beta=0$), this is equivalent to the uniform Oberwolfach problem and is covered by Theorem~\ref{OP uniform}; we thus assume that $n>m$. Except possibly when $\beta=1$, the cases $(m,n) = (3,5)$, $(3,15)$ are solved in \cite{AdaBilBryElz}, and the cases $(m,n) = (3,7), (3,9)$ are solved in \cite{Lei Fu} and \cite{Kamin}, respectively. So we may assume $(m,n) \neq (3,5), (3,7)$, $(3,9)$ or $(3,15)$.

We start with a $C_m$-factorization of $K_t[m]$, which exists by Theorem~\ref{Liu}. 
We note that this design has $r = m(t-1)/2$ $C_m$ factors. We expand each point of this factorization by $n$, to create a $C_m[n]$-factorization of $K_t[m][n]\cong K_t[mn]$ with $r$ $C_m[n]$-factors. This design has $t$ uncovered holes of size $mn$, which we can fill with a $C_m$ or $C_n$-factorization of $K_{mn}$ as we choose, by Theorem~\ref{OP uniform}. This step will yield either $(mn-1)/2$ $C_m$-factors or $(mn-1)/2$ $C_n$-factors of the whole design, as appropriate.
We can resolve each of the $r$ $C_m[n]$-factors into $\gamma_i$ $C_m$-factors and $\delta_i$ $C_n$-factors, where $\gamma_i + \delta_i = n$, where $\gamma=\sum_{i=1}^r\gamma_i$ is the number of classes left to fill with $C_m$-factors.

If $\alpha < (mn-1)/2$, set $\gamma=\alpha$ and fill the holes with $C_n$-factors, i.e.\ an HW$(mn; m,n; 0, (mn-1)/2)$.
Otherwise, when $\alpha \geq (mn-1)/2$, set $\gamma = \alpha - (mn-1)/2$ and fill the holes with $C_m$-factors, i.e.\ an HW$(mn; m, n; (mn-1)/2, 0)$.
Noting that $\gamma < mn(t-1)/2$, write $\gamma = xn + y$, with $0 \leq x < r$ and $0 \leq y < n$.

If the conditions of Theorem~\ref{C_m[n]} are satisfied for $y$, i.e.\ $(y,n-y) \in \mathrm{HWP}(C_m[n];m,n)$, we fill $x$ $C_m[n]$-factors with an HW$(C_m[n]; m, n;  n, 0)$, one $C_m[n]$-factor with an HW$(C_m[n]; m, n; y,n-y)$ and the rest with an HW$(C_m[n]; m, n; 0, n)$.

Otherwise, the conditions of Theorem~\ref{C_m[n]} are not satisfied for $y$, so that $y \in \{2,4,n-1,n-3\}$ or $(m,n,y)= (3,11,6), (3,13,8)$ (we are supposing $(m,n)\neq (3,15)$).
If $x \leq r-2$, we fill $x$ $C_m[n]$-factors with an HW$(C_m[n]; m, n; n, 0)$, one with an HW$(C_m[n]; m, n; 1,n-1)$, one with an HW$(C_m[n]; m, n; y-1, n-y+1)$ and the remaining $r-x-2$ with an $\mathrm{HW}(C_m[n];m,n;0,n)$.

Otherwise, we have that $x=r-1$.  Note that since $\beta \neq 1, 3$ by assumption, it follows that $y \neq n-1, n-3$, and so $y\in \{2,4\}$, or $(m,n,y) = (3,11,6), (3,13,8)$.
If $n \geq 11$ and $y = 2,4$, then we fill $x-1$ $C_m[n]$-factors with an HW$(C_m[n]; m, n; n, 0)$, one with an HW$(C_m[n]; m, n; n-5,5)$ (Theorem~\ref{C_m[n]}, $n-5>4$) and one with an HW$(C_m[n]; m, n; y+5, n-y-5)$ (Theorem~\ref{C_m[n]}, $y+5 = 7,9$).
Since we are assuming $(m,n) \neq (3,5)$, $(3,7)$, $(3,9)$, we are left with the case where $x=r-1$ and $(m,n,y) = (3,11,6), (3,13,8)$, or $(m, n) =  (5,7), (5,9), (7,9)$ and $y = 2,4$.
These correspond to the cases $(m,n,\beta) = (3,11,5), (3,13,5)$, or $(m, n) =  (5,7), (5,9), (7,9)$ and $\beta = n-2$ or $n-4$.

When $(m,n,\beta) = (5,7,3), (5,7,5)$ or $(3,11,5)$, $\beta = (n-1)/2$ or $(n+3)/2$, and so there exists an HW$(mn; m, n; (mn-1)/2-\beta,\beta)$ by Theorem~\ref{HW t=1}. We use this design to fill the $t$ uncovered holes of size $mn$ of the original $C_m$-factorization of $K_t[m]$, instead of a $C_m$-factorization. Now, we fill the expanded classes with $\mathrm{HW}(C_m[n];m,n;n,0)$.
This leaves the exceptions $(m,n,\beta) = (5,9,5), (5,9,7), (7,9,5), (7,9,7)$ and $(3,13,5)$.
\end{proof}

The results of Theorems~\ref{HW t=1} and \ref{HW t>1} together prove Theorem~\ref{main}.

In the case that the cycle lengths $m$ and $n$ are relatively prime, in order for $\mathrm{HW}(v;m,n;\alpha,\beta)$ to exist, it is necessary that $v$ is a multiple of $mn$.  We thus have the following corollary.
\begin{cor}
Let $m$ and $n$ be coprime odd integers with $n \geq m \geq 3$.  If $v \geq 3$ is an odd integer, then $(\alpha,\beta) \in \mathrm{HWP}(v;m,n)$ if and only if $v$ is divisible by $mn$, $\alpha \geq 0$ and $\beta \geq 0$, and $\alpha+\beta = (mnt-1)/2$, except possibly if:
\begin{itemize}
 \item $t>1$, and $\beta \in \{1,3\}$ or $(m,n,\beta) = (5,9,5)$, $(5,9,7)$, $(7,9,5)$, $(7,9,7)$, $(3,13,5)$. 
\item $t=1$ and one of the following conditions hold:  $\beta \in [1, \ldots, \frac{n-3}{2}] \cup \{\frac{n+1}{2}, \frac{n+5}{2}\}$, $(m,\alpha)=(3,2), (3,4)$, or $(m,n,\alpha,\beta) = (3,11,6,10)$, $(3,13,8,10)$, $(5,7,9,8)$, $(5,9,11,11)$, $(5,9,13,9)$, $(7,9,20,11)$, $(7,9,22,9)$.
\end{itemize}
\end{cor}

Finally, we note that by not filling the holes in the proof of Theorem~\ref{HW t>1} above, we have the following result on factorization of the complete multipartite graph.
\begin{cor}
\label{multipartite}
Let $n$ and $m$ be odd integers with $n\geq m \geq 3$.  Then $(\alpha, \beta)\in {\rm HWP}(K_t[mn]; m, n)$ if and only if $t>1$ is odd, $\alpha, \beta \geq 0$ and $\alpha+\beta = mn(t-1)/2$, except possibly if $\beta = 1$ or $3$ or  $(m,n,\beta) = (5,7,3)$, $(5,7,5)$ $(5,9,5)$, $(5,9,7)$, $(7,9,5)$, $(7,9,7)$, $(3,11,5)$, $(3,13,5)$.
\end{cor}

In a similar manner we may obtain solutions to HWP$(K_m[n]; m,n)$ from Theorem~\ref{HW t=1} by not filling in the parts in that Theorem.
We can then obtain further multipartite results on $K_{tm}[n]$ by filling in the parts in Corollary~\ref{multipartite} with these factorizations.

\section{Conclusion}

The Hamilton-Waterloo problem has received much attention over the last few years. While progress for pairs of odd factors has been slow up to now, we have shown that solutions exist for many cases. Indeed, if $m$ and $n$ are coprime, Theorem \ref{HW t>1} shows sufficiency when $v>mn$, $\beta \neq 1,3$ with one possible exception. For the case $v=mn$, Theorem~\ref{HW t=1} shows a similar result, but leaves the cases where $\beta < (n-3)/2$, $(n+1)/2$, $(n+5)/2$ as well as $\alpha \in \{2,4\}$ when $m=3$, with a few other small exceptions. It would be nice to see these exceptions closed. Of particular note is the case $\beta=1$; Theorem~\ref{beta=1} shows that in general the method that we have used, factoring $C_m[n]$, will not be usable in this case. We note that in the case of even cycles there is a similar problem when there is only one factor of a given type.

Whilst we have shown sufficiency when $m$ and $n$ are coprime, when  $m$ and $n$ are not coprime it is possible to have it is possible to have solutions where the number of points, $v$, is not a multiple of $mn$. Specifically, if $\ell = $lcm$(m,n)$, then in order for an $HW(v;m,n)$ to exist, $v$ must be a multiple of $\ell$. Except when $\beta=1,3$ or $v=mn$, we have covered every case where $v$ is a multiple of $mn$. This leaves the cases when $v$ a multiple of $\ell$, but is not divisible by $mn$.
Investigation of these cases would be of interest. 

For odd length cycles, we have only considered uniform factors. Almost nothing is known for non-uniform factors, in stark contrast to the even case.
The Hamilton-Waterloo problem also remains largely open when the factors have opposite parity. In addition, there is the more general Oberwolfach problem, when there are more than two factor types.

Theorem~\ref{C_m[n]} is a very important result in its own right. These factorizations are the cycle equivalent of group divisible designs in the design context and are likely to prove useful in solving a wide range of cycle factorization problems.

\section{Acknowledgements}

The first and second authors gratefully acknowledge the support of the NSERC Discovery grant program.

\begin{appendices}
\section{{Factoring} $C_m[n]$, $\alpha\equiv 2\pmod{6}$, $26 < \alpha < 224$}
\label{2 mod 6 Appendix}

For each value of $\alpha\equiv 2\pmod{6}$, with $26 < \alpha < 170$, we give $a$ and $b$ and $(\alpha-8)/6$ triples which, together with their negatives, satisfy the conditions of Lemma~\ref{gen 2}.\\

{\small
\noindent
{\boldmath $\alpha = 32$}, $a = 4$, $b = 1$, 
$\{ 3, 11, -14 \}$, $\{ 5, 13, -18 \}$, $\{ 6, 9, -15 \}$, $\{ 7, 10, -17 \}$. 

\noindent
{\boldmath $\alpha = 38$}, $a = 5$, $b = 1$,  
$\{ 3, 16, -19 \}$, $\{ 4, 13, -17 \}$, $\{ 6, 8, -14 \}$, $\{ 7, 11, -18 \}$, $\{ 9, 12, -21 \}$. 

\noindent
{\boldmath $\alpha = 44$}, $a = 3$, $b = 8$, 
$\{ 1, 13, -14 \}$, $\{ 2, 17, -19 \}$, $\{ 4, 20, -24 \}$, $\{ 5, 18, -23 \}$, $\{ 7, 15, -22 \}$, $\{ 10, 11, -21 \}$. 

\noindent
{\boldmath $\alpha = 50$}, $a = 3$, $b = 8$, 
$\{ 1, 14, -15 \}$, $\{ 2, 20, -22 \}$, $\{ 4, 19, -23 \}$, $\{ 5, 21, -26 \}$, $\{ 7, 18, -25 \}$, $\{ 10, 17, -27 \}$, $\{ 11, 13, -24 \}$. 

\noindent
{\boldmath $\alpha = 56$}, $a = 5$, $b = 1$, 
$\{ 3, 14, -17 \}$, $\{ 4, 21, -25 \}$, $\{ 6, 23, -29 \}$, $\{ 7, 19, -26 \}$, $\{ 8, 22, -30 \}$, $\{ 9, 18, -27 \}$, $\{ 11, 13, -24 \}$, $\{ 12, 16, -28 \}$. 

\noindent
{\boldmath $\alpha = 62$}, $a = 5$, $b = 8$, 
$\{ 1, 6, -7 \}$, $\{ 2, 24, -26 \}$, $\{ 3, 25, -28 \}$, $\{ 4, 23, -27 \}$, $\{ 9, 21, -30 \}$, $\{ 11, 18, -29 \}$, $\{ 12, 22, -34 \}$, $\{ 13, 19, -32 \}$, $\{ 14, 17, -31 \}$. 

\noindent
{\boldmath $\alpha = 68$}, $a = 5$, $b = 2$, 
$\{ 17, -31, 14 \}$, $\{ -30, 11, 19 \}$, $\{ 6, 18, -24 \}$, $\{ 29, -32, 3 \}$, $\{ -35, 8, 27 \}$, $\{ 9, 25, -34 \}$, $\{ 26, -33, 7 \}$, $\{ -28, 12, 16 \}$, $\{ -36, 13, 23 \}$, $\{ 1, 21, -22 \}$. 

\noindent
{\boldmath $\alpha = 74$}, $a = 5$, $b = 8$, 
$\{ 1, 11, -12 \}$, $\{ 2, 24, -26 \}$, $\{ 3, 28, -31 \}$, $\{ 4, 29, -33 \}$, $\{ 6, 30, -36 \}$, $\{ 7, 27, -34 \}$, $\{ 9, 23, -32 \}$, $\{ 13, 22, -35 \}$, $\{ 14, 25, -39 \}$, $\{ 17, 21, -38 \}$, $\{ 18, 19, -37 \}$. 

\noindent
{\boldmath $\alpha = 80$}, $a = 3$, $b = 1$,  
$\{ 23, 14, -37 \}$, $\{ -42, 25, 17 \}$, $\{ 11, 24, -35 \}$, $\{ 7, 27, -34 \}$, $\{ 26, 4, -30 \}$, $\{ -31, 16, 15 \}$, $\{ 13, 19, -32 \}$, $\{ 8, -36, 28 \}$, $\{ -41, 20, 21 \}$, $\{ 18, 22, -40 \}$, $\{ 29, -39, 10 \}$, $\{ 33, -38, 5 \}$. 

\noindent
{\boldmath $\alpha = 86$}, $a = 3$, $b = 1$, 
$\{ 20, -41, 21 \}$, $\{ -39, 25, 14 \}$, $\{ 22, -40, 18 \}$, $\{ 15, 30, -45 \}$, $\{ 13, 31, -44 \}$, $\{ 35, -43, 8 \}$, $\{ -36, 32, 4 \}$, $\{ 24, -29, 5 \}$, $\{ -42, 19, 23 \}$, $\{ 26, -37, 11 \}$, $\{ 27, -34, 7 \}$, $\{ 17, -33, 16 \}$, $\{ 28, -38, 10 \}$. 

\noindent
{\boldmath $\alpha = 92$}, $a = 3$, $b = 2$, 
$\{ 37, -38, 1 \}$, $\{ 31, -36, 5 \}$, $\{ -48, 34, 14 \}$, $\{ -46, 33, 13 \}$, $\{ -39, 11, 28 \}$, $\{ -44, 18, 26 \}$, $\{ 17, 23, -40 \}$, $\{ 35, 7, -42 \}$, $\{ 30, 15, -45 \}$, $\{ 8, 21, -29 \}$, $\{ 19, 24, -43 \}$, $\{ -47, 27, 20 \}$, $\{ -41, 16, 25 \}$, $\{ 22, 10, -32 \}$. 

\noindent
{\boldmath $\alpha = 98$}, $a = 7$, $b = 16$, 
$\{ 38, -48, 10 \}$, $\{ 43, 4, -47 \}$, $\{ 6, 12, -18 \}$, $\{ 30, 3, -33 \}$, $\{ 19, 22, -41 \}$, $\{ 8, 23, -31 \}$, $\{ 34, 17, -51 \}$, $\{ 9, -49, 40 \}$, $\{ 26, -39, 13 \}$, $\{ 20, -44, 24 \}$, $\{ 1, 45, -46 \}$, $\{ 15, 35, -50 \}$, $\{ 37, -42, 5 \}$, $\{ 27, 2, -29 \}$, $\{ 11, 25, -36 \}$. 

\noindent
{\boldmath $\alpha = 104$}, $a = 5$, $b = 1$,  
$\{ 32, -53, 21 \}$, $\{ 11, 27, -38 \}$, $\{ -39, 30, 9 \}$, $\{ 13, -46, 33 \}$, $\{ 19, -43, 24 \}$, $\{ -47, 25, 22 \}$, $\{ -40, 28, 12 \}$, $\{ 14, 17, -31 \}$, $\{ -54, 36, 18 \}$, $\{ 34, 3, -37 \}$, $\{ 23, -49, 26 \}$, $\{ -51, 7, 44 \}$, $\{ 16, 29, -45 \}$, $\{ 8, -50, 42 \}$, $\{ 35, -41, 6 \}$, $\{ 48, -52, 4 \}$. 

\noindent
{\boldmath $\alpha = 110$}, $a = 3$, $b = 1$, 
$\{ 20, 35, -55 \}$, $\{ 17, -45, 28 \}$, $\{ 30, -52, 22 \}$, $\{ 8, 21, -29 \}$, $\{ 11, 25, -36 \}$, $\{ 5, 38, -43 \}$, $\{ 23, 27, -50 \}$, $\{ -48, 16, 32 \}$, $\{ 26, 31, -57 \}$, $\{ 41, -56, 15 \}$, $\{ 34, 19, -53 \}$, $\{ 40, -54, 14 \}$, $\{ 33, -46, 13 \}$, $\{ 18, -42, 24 \}$, $\{ 10, 39, -49 \}$, $\{ 37, -44, 7 \}$, $\{ 47, -51, 4 \}$. 

\noindent
{\boldmath $\alpha = 116$}, $a = 7$, $b = 16$,  
$\{ 46, 3, -49 \}$, $\{ -58, 48, 10 \}$, $\{ 39, 4, -43 \}$, $\{ 41, -50, 9 \}$, $\{ 42, -59, 17 \}$, $\{ 36, 18, -54 \}$, $\{ 13, 20, -33 \}$, $\{ -55, 44, 11 \}$, $\{ -60, 15, 45 \}$, $\{ 12, 35, -47 \}$, $\{ 6, 34, -40 \}$, $\{ 26, 31, -57 \}$, $\{ 30, -38, 8 \}$, $\{ 19, 37, -56 \}$, $\{ 2, 23, -25 \}$, $\{ 51, -52, 1 \}$, $\{ 24, 29, -53 \}$, $\{ 5, 22, -27 \}$. 

\noindent
{\boldmath $\alpha = 122$}, $a = 7$, $b = 16$,  
$\{ 61, 22, -39 \}$, $\{ 50, -63, 13 \}$, $\{ 38, 9, -47 \}$, $\{ 18, 35, -53 \}$, $\{ 1, 4, -5 \}$, $\{ 23, 25, -48 \}$, $\{ -45, 34, 11 \}$, $\{ 52, -54, 2 \}$, $\{ -46, 26, 20 \}$, $\{ 6, -62, 56 \}$, $\{ 19, 36, -55 \}$, $\{ 33, 10, -43 \}$, $\{ 31, -60, 29 \}$, $\{ -51, 27, 24 \}$, $\{ 41, -58, 17 \}$, $\{ 44, -59, 15 \}$, $\{ 12, 30, -42 \}$, $\{ 49, -57, 8 \}$, $\{ 37, 3, -40 \}$. 

\noindent
{\boldmath $\alpha = 128$}, $a = 3$, $b = 1$,  
$\{ 4, 33, -37 \}$, $\{ 35, 29, -64 \}$, $\{ -52, 25, 27 \}$, $\{ -51, 10, 61 \}$, $\{ 17, 32, -49 \}$, $\{ 15, -54, 39 \}$, $\{ -55, 14, 41 \}$, $\{ 40, -60, 20 \}$, $\{ -44, 31, 13 \}$, $\{ -45, 22, 23 \}$, $\{ -63, 16, 47 \}$, $\{ 11, 42, -53 \}$, $\{ 7, 43, -50 \}$, $\{ -56, 26, 30 \}$, $\{ 57, -62, 5 \}$, $\{ 34, 24, -58 \}$, $\{ 46, -65, 19 \}$, $\{ -66, 48, 18 \}$, $\{ 21, 38, -59 \}$, $\{ 8, 28, -36 \}$. 

\noindent
{\boldmath $\alpha = 134$}, $a = 3$, $b = 1$, 
$\{ 5, -33, 38 \}$, $\{ 24, 30, -54 \}$, $\{ 49, -53, 4 \}$, $\{ 52, 11, -63 \}$, $\{ 18, 21, -39 \}$, $\{ 15, -59, 44 \}$, $\{ 28, -55, 27 \}$, $\{ 13, 34, -47 \}$, $\{ 37, 19, -56 \}$, $\{ -62, 45, 17 \}$, $\{ 41, 7, -48 \}$, $\{ 40, -69, 29 \}$, $\{ 22, 43, -65 \}$, $\{ 25, 32, -57 \}$, $\{ 10, 51, -61 \}$, $\{ -64, 14, 50 \}$, $\{ 8, 60, -68 \}$, $\{ 20, -66, 46 \}$, $\{ 23, 35, -58 \}$, $\{ 31, 36, -67 \}$, $\{ 16, 26, -42 \}$. 

\noindent
{\boldmath $\alpha = 140$}, $a = 7$, $b = 16$, 
$\{ 22, -52, 30 \}$, $\{ -54, 41, 13 \}$, $\{ 15, -62, 47 \}$, $\{ 20, 45, -65 \}$, $\{ 53, -57, 4 \}$, $\{ 37, -49, 12 \}$, $\{ 55, 1, -56 \}$, $\{ 46, -63, 17 \}$, $\{ 35, -58, 23 \}$, $\{ 25, 26, -51 \}$, $\{ 2, 27, -29 \}$, $\{ 66, -71, 5 \}$, $\{ 34, -70, 36 \}$, $\{ -67, 48, 19 \}$, $\{ 50, -61, 11 \}$, $\{ 42, 18, -60 \}$, $\{ 59, -68, 9 \}$, $\{ -64, 24, 40 \}$, $\{ 69, 3, -72 \}$, $\{ 33, -43, 10 \}$, $\{ 38, -44, 6 \}$, $\{ 31, 8, -39 \}$. 

\noindent
{\boldmath $\alpha = 146$}, $a = 7$, $b = 16$, 
$\{ 38, -73, 35 \}$, $\{ 55, -60, 5 \}$, $\{ 62, -68, 6 \}$, $\{ 39, 33, -72 \}$, $\{ -69, 24, 45 \}$, $\{ 30, 40, -70 \}$, $\{ -74, 11, 63 \}$, $\{ 27, 31, -58 \}$, $\{ 18, 43, -61 \}$, $\{ 46, -75, 29 \}$, $\{ 17, 25, -42 \}$, $\{ 20, 2, -22 \}$, $\{ -67, 13, 54 \}$, $\{ 23, 36, -59 \}$, $\{ -71, 19, 52 \}$, $\{ 49, 15, -64 \}$, $\{ 12, 44, -56 \}$, $\{ 8, 26, -34 \}$, $\{ 4, 53, -57 \}$, $\{ 10, 37, -47 \}$, $\{ 65, -66, 1 \}$, $\{ 3, 48, -51 \}$, $\{ 9, 41, -50 \}$. 

\noindent
{\boldmath $\alpha = 152$}, $a = 3$, $b = 1$, 
$\{ 16, 55, -71 \}$, $\{ 35, 37, -72 \}$, $\{ -78, 18, 60 \}$, $\{ 59, 8, -67 \}$, $\{ 50, -57, 7 \}$, $\{ 21, 40, -61 \}$, $\{ 43, -48, 5 \}$, $\{ 11, 63, -74 \}$, $\{ 14, 32, -46 \}$, $\{ -75, 44, 31 \}$, $\{ 49, 24, -73 \}$, $\{ 19, 4, -23 \}$, $\{ 41, 27, -68 \}$, $\{ -70, 42, 28 \}$, $\{ 54, 10, -64 \}$, $\{ 26, 39, -65 \}$, $\{ 22, 47, -69 \}$, $\{ 13, 45, -58 \}$, $\{ 25, 52, -77 \}$, $\{ 51, 34, -17 \}$, $\{ 56, -76, 20 \}$, $\{ 36, 30, -66 \}$, $\{ 33, 29, -62 \}$, $\{ 15, 38, -53 \}$. 

\noindent
{\boldmath $\alpha = 158$}, $a = 4$, $b = 1$, 
$\{ 22, 69, 47 \}$, $\{ 70, 79, 9 \}$, $\{ 20, 32, 52 \}$, $\{ 41, 71, 30 \}$, $\{ 17, 72, 55 \}$, $\{ 35, 38, 73 \}$, $\{ 63, 42, 21 \}$, $\{ 19, 48, 67 \}$, $\{ 40, 54, 14 \}$, $\{ 46, 57, 11 \}$, $\{ 56, 66, 10 \}$, $\{ 26, 13, 39 \}$, $\{ 64, 27, 37 \}$, $\{ 36, 45, 81 \}$, $\{ 49, 25, 74 \}$, $\{ 75, 31, 44 \}$, $\{ 51, 29, 80 \}$, $\{ 33, 43, 76 \}$, $\{ 24, 34, 58 \}$, $\{ 62, 77, 15 \}$, $\{ 5, 23, 28 \}$, $\{ 7, 61, 68 \}$, $\{ 18, 60, 78 \}$, $\{ 3, 50, 53 \}$, $\{ 6, 59, 65 \}$. 

\noindent
{\boldmath $\alpha = 164$}, $a = 7$, $b = 16$,  
$\{ 52, 20, -72 \}$, $\{ 76, -84, 8 \}$, $\{ 36, 42, -78 \}$, $\{ 69, -82, 13 \}$, $\{ 33, 47, -80 \}$, $\{ 29, 45, -74 \}$, $\{ 34, 39, -73 \}$, $\{ 79, -81, 2 \}$, $\{ 41, 9, -50 \}$, $\{ 37, 46, -83 \}$, $\{ -75, 31, 44 \}$, $\{ 24, 61, -85 \}$, $\{ 22, 35, -57 \}$, $\{ -70, 19, 51 \}$, $\{ 48, -60, 12 \}$, $\{ 4, 23, -27 \}$, $\{ -71, 54, 17 \}$, $\{ 26, -64, 38 \}$, $\{ 53, 3, -56 \}$, $\{ 55, -66, 11 \}$, $\{ 25, 43, -68 \}$, $\{ 58, 5, -63 \}$, $\{ 6, 59, -65 \}$, $\{ 10, 30, -40 \}$, $\{ 15, 62, -77 \}$, $\{ 49, -67, 18 \}$.
}

\noindent
{\boldmath $\alpha = 170$}, $a = 3$, $b = 16$,
$\{ 33, 46, 79 \}$, $\{ 66, 40, 26 \}$, $\{ 28, 47, 75 \}$, $\{ 25, 30, 55 \}$, $\{ 61, 7, 68 \}$, $\{ 20, 52, 72 \}$, $\{ 4, 38, 42 \}$, $\{ 53, 82, 29 \}$, $\{ 63, 74, 11 \}$, $\{ 1, 84, 85 \}$, $\{ 39, 44, 83 \}$, $\{ 70, 10, 80 \}$, $\{ 34, 23, 57 \}$, $\{ 13, 78, 65 \}$, $\{ 19, 2, 21 \}$, $\{ 51, 18, 69 \}$, $\{ 50, 81, 31 \}$, $\{ 86, 37, 49 \}$, $\{ 60, 87, 27 \}$, $\{ 36, 41, 77 \}$, $\{ 22, 45, 67 \}$, $\{ 24, 35, 59 \}$, $\{ 15, 58, 73 \}$, $\{ 43, 5, 48 \}$, $\{ 54, 71, 17 \}$, $\{ 76, 14, 62 \}$, $\{ 56, 8, 64 \}$.

\noindent
{\boldmath $\alpha = 176$}, $a = 4$, $b = 1$,
$\{ 7, 24, 31 \}$, $\{ 14, 64, 78 \}$, $\{ 83, 20, 63 \}$, $\{ 81, 9, 90 \}$, $\{ 42, 43, 85 \}$, $\{ 67, 13, 80 \}$, $\{ 18, 57, 75 \}$, $\{ 45, 29, 74 \}$, $\{ 32, 37, 69 \}$, $\{ 82, 22, 60 \}$, $\{ 26, 33, 59 \}$, $\{ 88, 58, 30 \}$, $\{ 76, 86, 10 \}$, $\{ 44, 55, 11 \}$, $\{ 51, 3, 54 \}$, $\{ 56, 41, 15 \}$, $\{ 87, 34, 53 \}$, $\{ 62, 79, 17 \}$, $\{ 73, 48, 25 \}$, $\{ 65, 27, 38 \}$, $\{ 84, 61, 23 \}$, $\{ 6, 46, 52 \}$, $\{ 35, 36, 71 \}$, $\{ 89, 50, 39 \}$, $\{ 21, 49, 70 \}$, $\{ 68, 28, 40 \}$, $\{ 72, 5, 77 \}$, $\{ 47, 66, 19 \}$.

\noindent
{\boldmath $\alpha = 182$}, $a = 3$, $b = 1$,
$\{ 40, 52, 92 \}$, $\{ 73, 8, 81 \}$, $\{ 28, 50, 78 \}$, $\{ 55, 90, 35 \}$, $\{ 18, 25, 43 \}$, $\{ 46, 23, 69 \}$, $\{ 61, 24, 85 \}$, $\{ 14, 70, 84 \}$, $\{ 5, 67, 72 \}$, $\{ 31, 34, 65 \}$, $\{ 29, 51, 80 \}$, $\{ 58, 13, 71 \}$, $\{ 47, 79, 32 \}$, $\{ 74, 36, 38 \}$, $\{ 17, 22, 39 \}$, $\{ 26, 62, 88 \}$, $\{ 19, 56, 75 \}$, $\{ 76, 27, 49 \}$, $\{ 16, 41, 57 \}$, $\{ 33, 20, 53 \}$, $\{ 15, 68, 83 \}$, $\{ 48, 59, 11 \}$, $\{ 42, 86, 44 \}$, $\{ 77, 10, 87 \}$, $\{ 93, 30, 63 \}$, $\{ 54, 91, 37 \}$, $\{ 60, 4, 64 \}$, $\{ 82, 7, 89 \}$, $\{ 21, 45, 66 \}$.

\noindent
{\boldmath $\alpha = 188$}, $a = 3$, $b = 2$, 
$\{ 10, 17, 27 \}$, $\{ 80, 13, 93 \}$, $\{ 43, 26, 69 \}$, $\{ 41, 96, 55 \}$, $\{ 60, 52, 8 \}$, $\{ 68, 18, 86 \}$, $\{ 42, 76, 34 \}$, $\{ 14, 59, 73 \}$, $\{ 94, 48, 46 \}$, $\{ 45, 61, 16 \}$, $\{ 23, 28, 51 \}$, $\{ 88, 67, 21 \}$, $\{ 81, 44, 37 \}$, $\{ 58, 87, 29 \}$, $\{ 31, 71, 40 \}$, $\{ 30, 79, 49 \}$, $\{ 35, 54, 89 \}$, $\{ 19, 83, 64 \}$, $\{ 57, 38, 95 \}$, $\{ 47, 25, 72 \}$, $\{ 65, 70, 5 \}$, $\{ 91, 7, 84 \}$, $\{ 78, 22, 56 \}$, $\{ 75, 36, 39 \}$, $\{ 32, 50, 82 \}$, $\{ 74, 85, 11 \}$, $\{ 77, 15, 92 \}$, $\{ 20, 33, 53 \}$, $\{ 66, 90, 24 \}$, $\{ 63, 62, 1 \}$.

\noindent
{\boldmath $\alpha = 194$}, $a = 3$, $b = 2$,
$\{ 5, 51, 56 \}$, $\{ 55, 71, 16 \}$, $\{ 73, 66, 7 \}$, $\{ 45, 80, 35 \}$, $\{ 23, 62, 85 \}$, $\{ 18, 19, 37 \}$, $\{ 96, 46, 50 \}$, $\{ 59, 30, 89 \}$, $\{ 17, 70, 87 \}$, $\{ 27, 40, 67 \}$, $\{ 75, 86, 11 \}$, $\{ 47, 83, 36 \}$, $\{ 79, 41, 38 \}$, $\{ 65, 26, 91 \}$, $\{ 53, 21, 74 \}$, $\{ 94, 81, 13 \}$, $\{ 72, 28, 44 \}$, $\{ 48, 90, 42 \}$, $\{ 10, 14, 24 \}$, $\{ 69, 77, 8 \}$, $\{ 33, 49, 82 \}$, $\{ 43, 52, 95 \}$, $\{ 88, 54, 34 \}$, $\{ 92, 29, 63 \}$, $\{ 58, 97, 39 \}$, $\{ 22, 76, 98 \}$, $\{ 84, 20, 64 \}$, $\{ 68, 31, 99 \}$, $\{ 60, 61, 1 \}$, $\{ 25, 32, 57 \}$, $\{ 78, 93, 15 \}$.

\noindent
{\boldmath $\alpha = 200$}, $a = 3$, $b = 1$
$\{ 77, 91, 14 \}$, $\{ 43, 49, 92 \}$, $\{ 83, 38, 45 \}$, $\{ 76, 23, 99 \}$, $\{ 90, 20, 70 \}$, $\{ 47, 88, 41 \}$, $\{ 42, 101, 59 \}$, $\{ 66, 85, 19 \}$, $\{ 71, 25, 96 \}$, $\{ 13, 54, 67 \}$, $\{ 98, 5, 93 \}$, $\{ 31, 50, 81 \}$, $\{ 56, 17, 73 \}$, $\{ 84, 55, 29 \}$, $\{ 62, 80, 18 \}$, $\{ 24, 65, 89 \}$, $\{ 28, 4, 32 \}$, $\{ 46, 68, 22 \}$, $\{ 35, 37, 72 \}$, $\{ 39, 100, 61 \}$, $\{ 74, 8, 82 \}$, $\{ 11, 86, 97 \}$, $\{ 57, 64, 7 \}$, $\{ 21, 27, 48 \}$, $\{ 10, 30, 40 \}$, $\{ 15, 60, 75 \}$, $\{ 78, 26, 52 \}$, $\{ 63, 16, 79 \}$, $\{ 95, 44, 51 \}$, $\{ 87, 34, 53 \}$, $\{ 69, 102, 33 \}$, $\{ 36, 94, 58 \}$.

\noindent
{\boldmath $\alpha = 206$}, $a = 3$, $b = 1$,
$\{ 38, 60, 98 \}$, $\{ 7, 54, 61 \}$, $\{ 97, 16, 81 \}$, $\{ 13, 24, 37 \}$, $\{ 99, 50, 49 \}$, $\{ 63, 55, 8 \}$, $\{ 64, 79, 15 \}$, $\{ 25, 43, 68 \}$, $\{ 83, 5, 88 \}$, $\{ 100, 18, 82 \}$, $\{ 104, 33, 71 \}$, $\{ 77, 87, 10 \}$, $\{ 30, 36, 66 \}$, $\{ 53, 4, 57 \}$, $\{ 45, 48, 93 \}$, $\{ 51, 86, 35 \}$, $\{ 40, 44, 84 \}$, $\{ 67, 27, 94 \}$, $\{ 21, 95, 74 \}$, $\{ 29, 62, 91 \}$, $\{ 17, 85, 102 \}$, $\{ 56, 78, 22 \}$, $\{ 11, 101, 90 \}$, $\{ 46, 65, 19 \}$, $\{ 69, 23, 92 \}$, $\{ 103, 31, 72 \}$, $\{ 34, 76, 42 \}$, $\{ 52, 28, 80 \}$, $\{ 58, 47, 105 \}$, $\{ 32, 41, 73 \}$, $\{ 26, 70, 96 \}$, $\{ 39, 20, 59 \}$, $\{ 75, 89, 14 \}$.

\noindent
{\boldmath $\alpha = 212$}, $a = 3$, $b = 2$,
$\{ 17, 42, 59 \}$, $\{ 56, 10, 66 \}$, $\{ 88, 24, 64 \}$, $\{ 1, 74, 75 \}$, $\{ 21, 80, 101 \}$, $\{ 55, 13, 68 \}$, $\{ 69, 96, 27 \}$, $\{ 92, 16, 108 \}$, $\{ 62, 103, 41 \}$, $\{ 51, 99, 48 \}$, $\{ 54, 19, 73 \}$, $\{ 46, 86, 40 \}$, $\{ 52, 77, 25 \}$, $\{ 94, 34, 60 \}$, $\{ 85, 20, 105 \}$, $\{ 97, 32, 65 \}$, $\{ 18, 5, 23 \}$, $\{ 90, 98, 8 \}$, $\{ 72, 61, 11 \}$, $\{ 38, 43, 81 \}$, $\{ 26, 58, 84 \}$, $\{ 37, 50, 87 \}$, $\{ 22, 71, 93 \}$, $\{ 106, 15, 91 \}$, $\{ 30, 70, 100 \}$, $\{ 76, 83, 7 \}$, $\{ 67, 95, 28 \}$, $\{ 39, 63, 102 \}$, $\{ 57, 104, 47 \}$, $\{ 107, 78, 29 \}$, $\{ 44, 79, 35 \}$, $\{ 53, 36, 89 \}$, $\{ 33, 49, 82 \}$, $\{ 14, 31, 45 \}$.

\noindent
{\boldmath $\alpha = 218$}, $a = 3$, $b = 2$,
$\{ 65, 105, 40 \}$, $\{ 35, 52, 87 \}$, $\{ 19, 84, 103 \}$, $\{ 96, 59, 37 \}$, $\{ 62, 30, 92 \}$, $\{ 69, 25, 94 \}$, $\{ 67, 24, 91 \}$, $\{ 15, 29, 44 \}$, $\{ 58, 39, 97 \}$, $\{ 32, 50, 82 \}$, $\{ 17, 100, 83 \}$, $\{ 108, 109, 1 \}$, $\{ 76, 42, 34 \}$, $\{ 104, 43, 61 \}$, $\{ 13, 101, 88 \}$, $\{ 99, 18, 81 \}$, $\{ 16, 38, 54 \}$, $\{ 36, 57, 93 \}$, $\{ 23, 45, 68 \}$, $\{ 63, 49, 14 \}$, $\{ 80, 90, 10 \}$, $\{ 20, 46, 66 \}$, $\{ 98, 28, 70 \}$, $\{ 106, 31, 75 \}$, $\{ 8, 71, 79 \}$, $\{ 26, 60, 86 \}$, $\{ 47, 27, 74 \}$, $\{ 78, 7, 85 \}$, $\{ 55, 111, 56 \}$, $\{ 53, 64, 11 \}$, $\{ 41, 48, 89 \}$, $\{ 73, 95, 22 \}$, $\{ 110, 33, 77 \}$, $\{ 102, 107, 5 \}$, $\{ 21, 51, 72 \}$.

\section{Factoring $C_m[n]$ $\alpha\equiv 4\pmod{6}$, $ 28 <\alpha < 268$}
\label{4 mod 6 Appendix}

For each value of $\alpha\equiv 4\pmod{6}$, with $28 < \alpha < 268$, we give $a$ and $b$ and $(\alpha-10)/6$ triples which, together with their negatives, satisfy the conditions of Lemma~\ref{gen 4}.

{\small
\noindent
{\boldmath $\alpha = 34$}, $a = 3$, $b = 8$, 
$\{ 1, 10, 11\}$, $\{ 2, 7, 9\}$, $\{ 4, 13, 17\}$, $\{ 5, 14, 19\}$. 

\noindent
{\boldmath $\alpha = 40$}, $a = 3$, $b = 1$, 
$\{ 4, 17, 21\}$, $\{ 5, 11, 16\}$, $\{ 7, 13, 20\}$, $\{ 8, 14, 22\}$, $\{ 9, 10, 19\}$. 

\noindent
{\boldmath $\alpha = 46$}, $a = 3$, $b = 1$, 
$\{ 4, 17, 21\}$, $\{ 5, 20, 25\}$, $\{ 7, 16, 23\}$, $\{ 8, 11, 19\}$, $\{ 9, 13, 22\}$, $\{ 10, 14, 24\}$. 

\noindent
{\boldmath $\alpha = 52$}, $a = 3$, $b = 2$, 
$\{ 1, 19, 20\}$, $\{ 5, 22, 27\}$, $\{ 7, 21, 28\}$, $\{ 8, 16, 24\}$, $\{ 9, 17, 26\}$, $\{ 10, 13, 23\}$, $\{ 11, 14, 25\}$. 

\noindent
{\boldmath $\alpha = 58$}, $a = 3$, $b = 2$, 
$\{ 1, 21, 22\}$, $\{ 5, 23, 28\}$, $\{ 7, 24, 31\}$, $\{ 8, 17, 25\}$, $\{ 9, 20, 29\}$, $\{ 10, 16, 26\}$, $\{ 11, 19, 30\}$, $\{ 13, 14, 27\}$. 

\noindent
{\boldmath $\alpha = 64$}, $a = 3$, $b = 1$, 
$\{ 4, 21, 25\}$, $\{ 5, 22, 27\}$, $\{ 7, 24, 31\}$, $\{ 8, 26, 34\}$, $\{ 9, 20, 29\}$, $\{ 10, 23, 33\}$, $\{ 11, 17, 28\}$, $\{ 13, 19, 32\}$, $\{ 14, 16, 30\}$. 

\noindent
{\boldmath $\alpha = 70$}, $a = 3$, $b = 1$, 
$\{ 4, 22, 26\}$, $\{ 5, 24, 29\}$, $\{ 7, 23, 30\}$, $\{ 8, 27, 35\}$, $\{ 9, 28, 37\}$, $\{ 10, 21, 31\}$, $\{ 11, 25, 36\}$, $\{ 13, 19, 32\}$, $\{ 14, 20, 34\}$, $\{ 16, 17, 33\}$. 

\noindent
{\boldmath $\alpha = 76$}, $a = 3$, $b = 2$, 
$\{ 13, 24, 37\}$, $\{ 14, 22, 36\}$, $\{ 31, 9, 40\}$, $\{ 17, 21, 38\}$, $\{ 10, 16, 26\}$, $\{ 35, 28, 7\}$, $\{ 29, 30, 1\}$, $\{ 39, 20, 19\}$, $\{ 34, 11, 23\}$, $\{ 8, 25, 33\}$, $\{ 27, 5, 32\}$. 

\noindent
{\boldmath $\alpha = 82$}, $a = 3$, $b = 2$, 
$\{ 27, 32, 5\}$, $\{ 19, 22, 41\}$, $\{ 42, 33, 9\}$, $\{ 34, 20, 14\}$, $\{ 35, 8, 43\}$, $\{ 17, 23, 40\}$, $\{ 24, 7, 31\}$, $\{ 25, 11, 36\}$, $\{ 10, 28, 38\}$, $\{ 26, 13, 39\}$, $\{ 21, 37, 16\}$, $\{ 29, 1, 30\}$. 

\noindent
{\boldmath $\alpha = 88$}, $a = 3$, $b = 1$, 
$\{ 46, 13, 33\}$, $\{ 17, 21, 38\}$, $\{ 39, 9, 30\}$, $\{ 45, 23, 22\}$, $\{ 19, 24, 43\}$, $\{ 14, 20, 34\}$, $\{ 42, 31, 11\}$, $\{ 28, 35, 7\}$, $\{ 41, 25, 16\}$, $\{ 10, 26, 36\}$, $\{ 8, 29, 37\}$, $\{ 27, 5, 32\}$, $\{ 40, 44, 4\}$.

\noindent
{\boldmath $\alpha = 94$}, $a = 3$, $b = 1$, 
$\{ 5, 32, 37\}$, $\{ 48, 25, 23\}$, $\{ 20, 26, 46\}$, $\{ 45, 31, 14\}$, $\{ 49, 13, 36\}$, $\{ 11, 28, 39\}$, $\{ 41, 24, 17\}$, $\{ 19, 21, 40\}$, $\{ 22, 30, 8\}$, $\{ 16, 27, 43\}$, $\{ 7, 35, 42\}$, $\{ 10, 44, 34\}$, $\{ 38, 47, 9\}$, $\{ 29, 33, 4\}$. 

\noindent
{\boldmath $\alpha = 100$}, $a = 3$, $b = 2$, 
$\{ 51, 38, 13\}$, $\{ 17, 25, 42\}$, $\{ 20, 23, 43\}$, $\{ 31, 47, 16\}$, $\{ 7, 29, 36\}$, $\{ 37, 11, 48\}$, $\{ 40, 14, 26\}$, $\{ 21, 28, 49\}$, $\{ 1, 33, 34\}$, $\{ 30, 22, 52\}$, $\{ 8, 24, 32\}$, $\{ 19, 27, 46\}$, $\{ 35, 10, 45\}$, $\{ 41, 50, 9\}$, $\{ 39, 44, 5\}$. 

\noindent
{\boldmath $\alpha = 106$}, $a = 3$, $b = 2$, 
$\{ 24, 29, 53\}$, $\{ 37, 13, 50\}$, $\{ 44, 54, 10\}$, $\{ 52, 19, 33\}$, $\{ 43, 5, 48\}$, $\{ 36, 45, 9\}$, $\{ 17, 22, 39\}$, $\{ 27, 7, 34\}$, $\{ 31, 47, 16\}$, $\{ 8, 32, 40\}$, $\{ 21, 35, 14\}$, $\{ 55, 30, 25\}$, $\{ 46, 20, 26\}$, $\{ 49, 11, 38\}$, $\{ 23, 51, 28\}$, $\{ 41, 1, 42\}$. 

\noindent
{\boldmath $\alpha = 112$}, $a = 3$, $b = 1$, 
$\{ 26, 29, 55\}$, $\{ 31, 48, 17\}$, $\{ 50, 57, 7\}$, $\{ 41, 49, 8\}$, $\{ 46, 51, 5\}$, $\{ 23, 37, 14\}$, $\{ 32, 52, 20\}$, $\{ 56, 40, 16\}$, $\{ 44, 33, 11\}$, $\{ 24, 30, 54\}$, $\{ 38, 42, 4\}$, $\{ 27, 36, 9\}$, $\{ 28, 53, 25\}$, $\{ 58, 39, 19\}$, $\{ 22, 43, 21\}$, $\{ 35, 45, 10\}$, $\{ 47, 13, 34\}$. 

\noindent
{\boldmath $\alpha = 118$}, $a = 3$, $b = 1$,
$\{ 41, 9, 50\}$, $\{ 4, 38, 42\}$, $\{ 5, 58, 53\}$, $\{ 31, 55, 24\}$, $\{ 43, 57, 14\}$, $\{ 49, 7, 56\}$, $\{ 27, 61, 34\}$, $\{ 35, 51, 16\}$, $\{ 46, 26, 20\}$, $\{ 30, 59, 29\}$, $\{ 8, 32, 40\}$, $\{ 13, 23, 36\}$, $\{ 22, 25, 47\}$, $\{ 37, 48, 11\}$, $\{ 17, 28, 45\}$, $\{ 60, 39, 21\}$, $\{ 54, 10, 44\}$, $\{ 52, 19, 33\}$. 

\noindent
{\boldmath $\alpha = 124$}, $a = 3$, $b = 4$, 
$\{ 49, 25, 24\}$, $\{ 10, 51, 61\}$, $\{ 44, 14, 58\}$, $\{ 11, 43, 54\}$, $\{ 17, 21, 38\}$, $\{ 22, 33, 55\}$, $\{ 30, 64, 34\}$, $\{ 52, 53, 1\}$, $\{ 23, 62, 39\}$, $\{ 47, 60, 13\}$, $\{ 50, 57, 7\}$, $\{ 27, 56, 29\}$, $\{ 31, 32, 63\}$, $\{ 20, 28, 48\}$, $\{ 9, 36, 45\}$, $\{ 19, 40, 59\}$, $\{ 5, 41, 46\}$, $\{ 26, 42, 16\}$, $\{ 2, 35, 37\}$. 

\noindent
{\boldmath $\alpha = 130$}, $a = 3$, $b = 2$, 
$\{ 46, 33, 13\}$, $\{ 19, 61, 42\}$, $\{ 30, 32, 62\}$, $\{ 29, 52, 23\}$, $\{ 51, 58, 7\}$, $\{ 1, 35, 36\}$, $\{ 26, 50, 24\}$, $\{ 55, 14, 41\}$, $\{ 54, 45, 9\}$, $\{ 21, 49, 28\}$, $\{ 39, 66, 27\}$, $\{ 40, 57, 17\}$, $\{ 16, 43, 59\}$, $\{ 8, 64, 56\}$, $\{ 53, 22, 31\}$, $\{ 60, 65, 5\}$, $\{ 11, 37, 48\}$, $\{ 63, 25, 38\}$, $\{ 10, 34, 44\}$, $\{ 67, 47, 20\}$. 

\noindent
{\boldmath $\alpha = 136$}, $a = 3$, $b = 1$, 
$\{ 44, 49, 5\}$, $\{ 55, 41, 14\}$, $\{ 53, 25, 28\}$, $\{ 21, 47, 68\}$, $\{ 13, 37, 50\}$, $\{ 39, 48, 9\}$, $\{ 19, 32, 51\}$, $\{ 10, 42, 52\}$, $\{ 62, 26, 36\}$, $\{ 61, 45, 16\}$, $\{ 11, 59, 70\}$, $\{ 31, 33, 64\}$, $\{ 7, 22, 29\}$, $\{ 35, 69, 34\}$, $\{ 67, 27, 40\}$, $\{ 38, 20, 58\}$, $\{ 57, 65, 8\}$, $\{ 17, 63, 46\}$, $\{ 24, 30, 54\}$, $\{ 43, 66, 23\}$, $\{ 4, 56, 60\}$. 

\noindent
{\boldmath $\alpha = 142$}, $a = 3$, $b = 1$, 
$\{ 31, 39, 70\}$, $\{ 45, 10, 55\}$, $\{ 67, 40, 27\}$, $\{ 5, 21, 26\}$, $\{ 16, 56, 72\}$, $\{ 50, 63, 13\}$, $\{ 30, 36, 66\}$, $\{ 68, 8, 60\}$, $\{ 57, 20, 37\}$, $\{ 4, 49, 53\}$, $\{ 7, 47, 54\}$, $\{ 22, 42, 64\}$, $\{ 32, 73, 41\}$, $\{ 51, 14, 65\}$, $\{ 17, 29, 46\}$, $\{ 9, 34, 43\}$, $\{ 61, 28, 33\}$, $\{ 23, 35, 58\}$, $\{ 25, 69, 44\}$, $\{ 24, 62, 38\}$, $\{ 11, 48, 59\}$, $\{ 52, 19, 71\}$. 

\noindent
{\boldmath $\alpha = 148$}, $a = 3$, $b = 4$, 
$\{ 64, 30, 34\}$, $\{ 71, 2, 73\}$, $\{ 19, 48, 67\}$, $\{ 59, 9, 68\}$, $\{ 25, 62, 37\}$, $\{ 43, 53, 10\}$, $\{ 11, 21, 32\}$, $\{ 14, 27, 41\}$, $\{ 51, 24, 75\}$, $\{ 44, 22, 66\}$, $\{ 38, 69, 31\}$, $\{ 5, 50, 55\}$, $\{ 20, 52, 72\}$, $\{ 76, 40, 36\}$, $\{ 23, 35, 58\}$, $\{ 28, 42, 70\}$, $\{ 17, 46, 63\}$, $\{ 26, 39, 65\}$, $\{ 29, 74, 45\}$, $\{ 16, 33, 49\}$, $\{ 54, 61, 7\}$, $\{ 60, 47, 13\}$, $\{ 56, 57, 1\}$,

\noindent
{\boldmath $\alpha = 154$}, $a = 3$, $b = 2$, 
$\{ 26, 45, 71\}$, $\{ 37, 39, 76\}$, $\{ 75, 16, 59\}$, $\{ 51, 17, 68\}$, $\{ 23, 50, 73\}$, $\{ 65, 25, 40\}$, $\{ 19, 53, 72\}$, $\{ 70, 29, 41\}$, $\{ 56, 42, 14\}$, $\{ 10, 74, 64\}$, $\{ 46, 20, 66\}$, $\{ 77, 44, 33\}$, $\{ 11, 27, 38\}$, $\{ 54, 63, 9\}$, $\{ 32, 60, 28\}$, $\{ 57, 5, 62\}$, $\{ 8, 61, 69\}$, $\{ 1, 34, 35\}$, $\{ 58, 79, 21\}$, $\{ 78, 47, 31\}$, $\{ 24, 43, 67\}$, $\{ 52, 30, 22\}$, $\{ 7, 48, 55\}$, $\{ 36, 49, 13\}$. 

\noindent
{\boldmath $\alpha = 160$}, $a = 3$, $b = 1$, 
$\{ 23, 58, 81\}$, $\{ 51, 28, 79\}$, $\{ 78, 61, 17\}$, $\{ 60, 5, 65\}$, $\{ 80, 35, 45\}$, $\{ 82, 10, 72\}$, $\{ 33, 21, 54\}$, $\{ 68, 32, 36\}$, $\{ 13, 37, 50\}$, $\{ 14, 55, 69\}$, $\{ 7, 59, 66\}$, $\{ 75, 46, 29\}$, $\{ 9, 47, 56\}$, $\{ 22, 49, 71\}$, $\{ 70, 62, 8\}$, $\{ 20, 64, 44\}$, $\{ 67, 40, 27\}$, $\{ 63, 25, 38\}$, $\{ 30, 11, 41\}$, $\{ 74, 43, 31\}$, $\{ 19, 76, 57\}$, $\{ 77, 53, 24\}$, $\{ 16, 26, 42\}$, $\{ 73, 34, 39\}$, $\{ 48, 4, 52\}$. 

\noindent
{\boldmath $\alpha = 166$}, $a = 3$, $b = 1$, 
$\{ 84, 47, 37\}$, $\{ 79, 57, 22\}$, $\{ 8, 50, 58\}$, $\{ 42, 51, 9\}$, $\{ 60, 16, 76\}$, $\{ 64, 83, 19\}$, $\{ 61, 34, 27\}$, $\{ 26, 45, 71\}$, $\{ 4, 63, 67\}$, $\{ 85, 23, 62\}$, $\{ 82, 29, 53\}$, $\{ 46, 81, 35\}$, $\{ 41, 66, 25\}$, $\{ 13, 17, 30\}$, $\{ 20, 52, 72\}$, $\{ 36, 38, 74\}$, $\{ 10, 39, 49\}$, $\{ 59, 70, 11\}$, $\{ 32, 33, 65\}$, $\{ 21, 77, 56\}$, $\{ 54, 78, 24\}$, $\{ 14, 55, 69\}$, $\{ 28, 40, 68\}$, $\{ 31, 75, 44\}$, $\{ 7, 80, 73\}$, $\{ 43, 48, 5\}$. 

\noindent
{\boldmath $\alpha = 172$}, $a = 3$, $b = 2$, 
$\{ 21, 32, 53\}$, $\{ 10, 20, 30\}$, $\{ 34, 80, 46\}$, $\{ 65, 28, 37\}$, $\{ 59, 45, 14\}$, $\{ 47, 8, 55\}$, $\{ 44, 71, 27\}$, $\{ 83, 31, 52\}$, $\{ 50, 16, 66\}$, $\{ 5, 49, 54\}$, $\{ 26, 62, 88\}$, $\{ 35, 42, 77\}$, $\{ 56, 29, 85\}$, $\{ 17, 70, 87\}$, $\{ 68, 69, 1\}$, $\{ 51, 76, 25\}$, $\{ 61, 72, 11\}$, $\{ 24, 39, 63\}$, $\{ 60, 22, 82\}$, $\{ 81, 23, 58\}$, $\{ 78, 40, 38\}$, $\{ 36, 79, 43\}$, $\{ 57, 7, 64\}$, $\{ 73, 13, 86\}$, $\{ 74, 33, 41\}$, $\{ 84, 9, 75\}$, $\{ 48, 19, 67\}$. 

\noindent
{\boldmath $\alpha = 178$}, $a = 3$, $b = 2$, 
$\{ 39, 40, 79\}$, $\{ 26, 47, 73\}$, $\{ 16, 90, 74\}$, $\{ 68, 20, 88\}$, $\{ 21, 41, 62\}$, $\{ 51, 19, 70\}$, $\{ 37, 38, 75\}$, $\{ 69, 64, 5\}$, $\{ 52, 65, 13\}$, $\{ 14, 81, 67\}$, $\{ 56, 57, 1\}$, $\{ 89, 59, 30\}$, $\{ 80, 34, 46\}$, $\{ 22, 71, 49\}$, $\{ 35, 42, 77\}$, $\{ 91, 25, 66\}$, $\{ 50, 82, 32\}$, $\{ 86, 23, 63\}$, $\{ 72, 83, 11\}$, $\{ 53, 8, 61\}$, $\{ 84, 60, 24\}$, $\{ 27, 28, 55\}$, $\{ 33, 78, 45\}$, $\{ 76, 9, 85\}$, $\{ 44, 10, 54\}$, $\{ 87, 29, 58\}$, $\{ 7, 36, 43\}$, $\{ 17, 31, 48\}$. 

\noindent
{\boldmath $\alpha = 184$}, $a = 3$, $b = 1$, 
$\{ 85, 34, 51\}$, $\{ 64, 13, 77\}$, $\{ 27, 93, 66\}$, $\{ 56, 80, 24\}$, $\{ 28, 45, 73\}$, $\{ 44, 60, 16\}$, $\{ 19, 52, 71\}$, $\{ 11, 14, 25\}$, $\{ 61, 84, 23\}$, $\{ 31, 47, 78\}$, $\{ 38, 88, 50\}$, $\{ 70, 41, 29\}$, $\{ 39, 4, 43\}$, $\{ 75, 26, 49\}$, $\{ 87, 8, 79\}$, $\{ 42, 72, 30\}$, $\{ 58, 91, 33\}$, $\{ 92, 55, 37\}$, $\{ 83, 35, 48\}$, $\{ 46, 36, 82\}$, $\{ 5, 54, 59\}$, $\{ 89, 69, 20\}$, $\{ 21, 65, 86\}$, $\{ 74, 7, 81\}$, $\{ 90, 22, 68\}$, $\{ 32, 62, 94\}$, $\{ 67, 76, 9\}$, $\{ 17, 40, 57\}$, $\{ 63, 53, 10\}$. 

\noindent
{\boldmath $\alpha = 190$}, $a = 3$, $b = 1$, 
$\{ 54, 13, 67\}$, $\{ 77, 85, 8\}$, $\{ 26, 65, 91\}$, $\{ 16, 25, 41\}$, $\{ 66, 70, 4\}$, $\{ 86, 97, 11\}$, $\{ 74, 5, 79\}$, $\{ 17, 44, 61\}$, $\{ 42, 47, 89\}$, $\{ 49, 63, 14\}$, $\{ 40, 76, 36\}$, $\{ 87, 57, 30\}$, $\{ 20, 48, 68\}$, $\{ 45, 88, 43\}$, $\{ 51, 21, 72\}$, $\{ 33, 71, 38\}$, $\{ 50, 78, 28\}$, $\{ 84, 9, 93\}$, $\{ 81, 59, 22\}$, $\{ 46, 75, 29\}$, $\{ 52, 62, 10\}$, $\{ 31, 95, 64\}$, $\{ 94, 55, 39\}$, $\{ 53, 80, 27\}$, $\{ 35, 69, 34\}$, $\{ 73, 23, 96\}$, $\{ 83, 90, 7\}$, $\{ 24, 82, 58\}$, $\{ 92, 60, 32\}$, $\{ 19, 37, 56\}$. 

\noindent
{\boldmath $\alpha = 196$}, $a = 3$, $b = 2$, 
$\{ 67, 9, 76\}$, $\{ 91, 69, 22\}$, $\{ 81, 43, 38\}$, $\{ 21, 44, 65\}$, $\{ 17, 45, 62\}$, $\{ 75, 24, 99\}$, $\{ 29, 56, 85\}$, $\{ 34, 54, 88\}$, $\{ 40, 46, 86\}$, $\{ 60, 92, 32\}$, $\{ 27, 73, 100\}$, $\{ 84, 95, 11\}$, $\{ 78, 42, 36\}$, $\{ 30, 79, 49\}$, $\{ 37, 57, 94\}$, $\{ 50, 1, 51\}$, $\{ 83, 28, 55\}$, $\{ 74, 16, 90\}$, $\{ 31, 33, 64\}$, $\{ 89, 96, 7\}$, $\{ 35, 13, 48\}$, $\{ 52, 14, 66\}$, $\{ 98, 39, 59\}$, $\{ 19, 58, 77\}$, $\{ 61, 41, 20\}$, $\{ 5, 82, 87\}$, $\{ 71, 97, 26\}$, $\{ 23, 47, 70\}$, $\{ 10, 53, 63\}$, $\{ 68, 93, 25\}$, $\{ 72, 80, 8\}$. 

\noindent
{\boldmath $\alpha = 202$}, $a = 3$, $b = 4$, 
$\{ 37, 62, 99\}$, $\{ 58, 63, 5\}$, $\{ 88, 13, 101\}$, $\{ 61, 95, 34\}$, $\{ 46, 19, 65\}$, $\{ 74, 76, 2\}$, $\{ 100, 49, 51\}$, $\{ 91, 39, 52\}$, $\{ 17, 85, 102\}$, $\{ 64, 16, 80\}$, $\{ 96, 29, 67\}$, $\{ 98, 26, 72\}$, $\{ 57, 82, 25\}$, $\{ 89, 66, 23\}$, $\{ 36, 68, 32\}$, $\{ 59, 86, 27\}$, $\{ 42, 1, 43\}$, $\{ 81, 11, 92\}$, $\{ 47, 97, 50\}$, $\{ 14, 56, 70\}$, $\{ 30, 54, 84\}$, $\{ 22, 38, 60\}$, $\{ 41, 94, 53\}$, $\{ 31, 40, 71\}$, $\{ 77, 10, 87\}$, $\{ 33, 45, 78\}$, $\{ 48, 21, 69\}$, $\{ 73, 93, 20\}$, $\{ 7, 83, 90\}$, $\{ 9, 35, 44\}$, $\{ 28, 75, 103\}$, $\{ 79, 55, 24\}$. 

\noindent
{\boldmath $\alpha = 208$}, $a = 3$, $b = 1$, 
$\{ 88, 31, 57\}$, $\{ 59, 81, 22\}$, $\{ 101, 23, 78\}$, $\{ 5, 94, 99\}$, $\{ 50, 19, 69\}$, $\{ 73, 97, 24\}$, $\{ 91, 44, 47\}$, $\{ 62, 34, 96\}$, $\{ 70, 77, 7\}$, $\{ 25, 83, 58\}$, $\{ 56, 65, 9\}$, $\{ 72, 8, 80\}$, $\{ 93, 40, 53\}$, $\{ 32, 66, 98\}$, $\{ 11, 92, 103\}$, $\{ 16, 106, 90\}$, $\{ 33, 74, 41\}$, $\{ 27, 95, 68\}$, $\{ 102, 42, 60\}$, $\{ 85, 20, 105\}$, $\{ 63, 46, 17\}$, $\{ 86, 51, 35\}$, $\{ 89, 76, 13\}$, $\{ 28, 82, 54\}$, $\{ 45, 100, 55\}$, $\{ 14, 38, 52\}$, $\{ 26, 61, 87\}$, $\{ 30, 49, 79\}$, $\{ 36, 48, 84\}$, $\{ 104, 67, 37\}$, $\{ 71, 75, 4\}$, $\{ 43, 21, 64\}$, $\{ 10, 29, 39\}$. 

\noindent
{\boldmath $\alpha = 214$}, $a = 3$, $b = 1$, 
$\{ 90, 36, 54\}$, $\{ 89, 94, 5\}$, $\{ 8, 51, 59\}$, $\{ 96, 49, 47\}$, $\{ 53, 76, 23\}$, $\{ 32, 68, 100\}$, $\{ 104, 34, 70\}$, $\{ 93, 56, 37\}$, $\{ 55, 44, 99\}$, $\{ 63, 29, 92\}$, $\{ 107, 64, 43\}$, $\{ 26, 75, 101\}$, $\{ 108, 21, 87\}$, $\{ 57, 46, 103\}$, $\{ 60, 24, 84\}$, $\{ 33, 48, 81\}$, $\{ 95, 82, 13\}$, $\{ 4, 73, 77\}$, $\{ 41, 42, 83\}$, $\{ 25, 10, 35\}$, $\{ 66, 52, 14\}$, $\{ 40, 62, 102\}$, $\{ 19, 91, 72\}$, $\{ 98, 27, 71\}$, $\{ 11, 74, 85\}$, $\{ 17, 22, 39\}$, $\{ 16, 45, 61\}$, $\{ 9, 79, 88\}$, $\{ 67, 38, 105\}$, $\{ 30, 80, 50\}$, $\{ 106, 86, 20\}$, $\{ 69, 97, 28\}$, $\{ 78, 109, 31\}$, $\{ 58, 65, 7\}$. 

\noindent
{\boldmath $\alpha = 220$}, $a = 3$, $b = 2$, 
$\{ 72, 13, 85\}$, $\{ 39, 61, 100\}$, $\{ 40, 73, 33\}$, $\{ 102, 36, 66\}$, $\{ 76, 17, 93\}$, $\{ 75, 16, 91\}$, $\{ 25, 63, 88\}$, $\{ 94, 108, 14\}$, $\{ 65, 107, 42\}$, $\{ 105, 35, 70\}$, $\{ 74, 21, 95\}$, $\{ 37, 59, 96\}$, $\{ 81, 101, 20\}$, $\{ 56, 110, 54\}$, $\{ 64, 90, 26\}$, $\{ 92, 97, 5\}$, $\{ 8, 47, 55\}$, $\{ 89, 23, 112\}$, $\{ 22, 38, 60\}$, $\{ 27, 82, 109\}$, $\{ 58, 104, 46\}$, $\{ 24, 28, 52\}$, $\{ 32, 111, 79\}$, $\{ 87, 106, 19\}$, $\{ 83, 34, 49\}$, $\{ 50, 53, 103\}$, $\{ 62, 51, 11\}$, $\{ 77, 7, 84\}$, $\{ 69, 29, 98\}$, $\{ 10, 57, 67\}$, $\{ 68, 99, 31\}$, $\{ 86, 41, 45\}$, $\{ 80, 71, 9\}$, $\{ 48, 78, 30\}$, $\{ 1, 43, 44\}$. 

\noindent
{\boldmath $\alpha = 226$}, $a = 3$, $b = 2$, 
$\{ 13, 27, 40\}$, $\{ 36, 57, 93\}$, $\{ 24, 53, 77\}$, $\{ 94, 102, 8\}$, $\{ 62, 22, 84\}$, $\{ 61, 30, 91\}$, $\{ 51, 56, 107\}$, $\{ 75, 41, 34\}$, $\{ 95, 10, 105\}$, $\{ 16, 92, 108\}$, $\{ 35, 114, 79\}$, $\{ 113, 74, 39\}$, $\{ 48, 55, 103\}$, $\{ 69, 76, 7\}$, $\{ 104, 59, 45\}$, $\{ 1, 28, 29\}$, $\{ 31, 89, 58\}$, $\{ 80, 20, 100\}$, $\{ 97, 14, 111\}$, $\{ 98, 25, 73\}$, $\{ 86, 32, 54\}$, $\{ 66, 17, 83\}$, $\{ 99, 47, 52\}$, $\{ 60, 110, 50\}$, $\{ 65, 23, 88\}$, $\{ 38, 43, 81\}$, $\{ 37, 64, 101\}$, $\{ 63, 5, 68\}$, $\{ 71, 44, 115\}$, $\{ 33, 49, 82\}$, $\{ 106, 21, 85\}$, $\{ 70, 112, 42\}$, $\{ 90, 109, 19\}$, $\{ 26, 46, 72\}$, $\{ 67, 11, 78\}$, $\{ 9, 87, 96\}$. 

\noindent
{\boldmath $\alpha = 232$}, $a = 3$, $b = 1$, 
$\{ 102, 54, 48\}$, $\{ 77, 42, 35\}$, $\{ 66, 16, 82\}$, $\{ 112, 45, 67\}$, $\{ 100, 37, 63\}$, $\{ 80, 41, 39\}$, $\{ 11, 107, 118\}$, $\{ 23, 49, 72\}$, $\{ 62, 22, 84\}$, $\{ 36, 59, 95\}$, $\{ 91, 83, 8\}$, $\{ 68, 13, 81\}$, $\{ 104, 47, 57\}$, $\{ 31, 86, 117\}$, $\{ 19, 74, 93\}$, $\{ 5, 98, 103\}$, $\{ 51, 65, 116\}$, $\{ 97, 14, 111\}$, $\{ 88, 58, 30\}$, $\{ 70, 90, 20\}$, $\{ 113, 87, 26\}$, $\{ 69, 52, 17\}$, $\{ 79, 27, 106\}$, $\{ 96, 53, 43\}$, $\{ 21, 25, 46\}$, $\{ 61, 24, 85\}$, $\{ 89, 10, 99\}$, $\{ 50, 28, 78\}$, $\{ 29, 44, 73\}$, $\{ 115, 40, 75\}$, $\{ 94, 56, 38\}$, $\{ 4, 110, 114\}$, $\{ 105, 34, 71\}$, $\{ 32, 92, 60\}$, $\{ 101, 108, 7\}$, $\{ 109, 76, 33\}$, $\{ 9, 55, 64\}$. 

\noindent
{\boldmath $\alpha = 238$}, $a = 3$, $b = 1$, 
$\{ 100, 47, 53\}$, $\{ 77, 34, 111\}$, $\{ 118, 39, 79\}$, $\{ 76, 30, 106\}$, $\{ 65, 23, 88\}$, $\{ 5, 69, 74\}$, $\{ 83, 22, 105\}$, $\{ 89, 121, 32\}$, $\{ 64, 33, 97\}$, $\{ 96, 38, 58\}$, $\{ 24, 61, 85\}$, $\{ 103, 10, 113\}$, $\{ 78, 8, 86\}$, $\{ 91, 56, 35\}$, $\{ 67, 27, 94\}$, $\{ 29, 110, 81\}$, $\{ 44, 57, 101\}$, $\{ 87, 71, 16\}$, $\{ 108, 82, 26\}$, $\{ 90, 40, 50\}$, $\{ 36, 48, 84\}$, $\{ 20, 60, 80\}$, $\{ 37, 109, 72\}$, $\{ 21, 31, 52\}$, $\{ 55, 114, 59\}$, $\{ 43, 116, 73\}$, $\{ 41, 92, 51\}$, $\{ 13, 107, 120\}$, $\{ 7, 68, 75\}$, $\{ 14, 49, 63\}$, $\{ 25, 70, 95\}$, $\{ 99, 45, 54\}$, $\{ 93, 102, 9\}$, $\{ 17, 11, 28\}$, $\{ 42, 62, 104\}$, $\{ 98, 117, 19\}$, $\{ 46, 112, 66\}$, $\{ 115, 4, 119\}$. 

\noindent
{\boldmath $\alpha = 244$}, $a = 3$, $b = 2$, 
$\{ 64, 41, 105\}$, $\{ 69, 10, 79\}$, $\{ 85, 113, 28\}$, $\{ 60, 102, 42\}$, $\{ 107, 31, 76\}$, $\{ 13, 74, 87\}$, $\{ 30, 65, 95\}$, $\{ 5, 17, 22\}$, $\{ 8, 110, 118\}$, $\{ 90, 83, 7\}$, $\{ 49, 50, 99\}$, $\{ 84, 46, 38\}$, $\{ 59, 35, 94\}$, $\{ 36, 97, 61\}$, $\{ 115, 72, 43\}$, $\{ 100, 20, 120\}$, $\{ 77, 106, 29\}$, $\{ 24, 47, 71\}$, $\{ 78, 52, 26\}$, $\{ 40, 68, 108\}$, $\{ 51, 53, 104\}$, $\{ 57, 66, 123\}$, $\{ 67, 44, 111\}$, $\{ 82, 109, 27\}$, $\{ 101, 117, 16\}$, $\{ 56, 88, 32\}$, $\{ 58, 103, 45\}$, $\{ 63, 86, 23\}$, $\{ 14, 34, 48\}$, $\{ 91, 25, 116\}$, $\{ 21, 75, 96\}$, $\{ 37, 92, 55\}$, $\{ 93, 19, 112\}$, $\{ 81, 114, 33\}$, $\{ 11, 62, 73\}$, $\{ 70, 54, 124\}$, $\{ 89, 9, 98\}$, $\{ 80, 119, 39\}$, $\{ 121, 122, 1\}$. 

\noindent
{\boldmath $\alpha = 250$}, $a = 3$, $b = 2$, 
$\{ 54, 59, 113\}$, $\{ 77, 104, 27\}$, $\{ 125, 60, 65\}$, $\{ 69, 78, 9\}$, $\{ 14, 61, 75\}$, $\{ 39, 96, 57\}$, $\{ 20, 95, 115\}$, $\{ 31, 91, 122\}$, $\{ 50, 64, 114\}$, $\{ 89, 97, 8\}$, $\{ 47, 56, 103\}$, $\{ 53, 105, 52\}$, $\{ 41, 126, 85\}$, $\{ 72, 30, 102\}$, $\{ 44, 86, 42\}$, $\{ 10, 73, 83\}$, $\{ 112, 36, 76\}$, $\{ 70, 40, 110\}$, $\{ 80, 51, 29\}$, $\{ 67, 34, 101\}$, $\{ 107, 17, 124\}$, $\{ 106, 74, 32\}$, $\{ 116, 1, 117\}$, $\{ 82, 127, 45\}$, $\{ 24, 108, 84\}$, $\{ 93, 98, 5\}$, $\{ 119, 71, 48\}$, $\{ 25, 37, 62\}$, $\{ 66, 79, 13\}$, $\{ 16, 19, 35\}$, $\{ 94, 26, 120\}$, $\{ 55, 68, 123\}$, $\{ 92, 7, 99\}$, $\{ 118, 28, 90\}$, $\{ 38, 49, 87\}$, $\{ 63, 46, 109\}$, $\{ 58, 23, 81\}$, $\{ 121, 88, 33\}$, $\{ 21, 22, 43\}$, $\{ 11, 100, 111\}$. 

\noindent
{\boldmath $\alpha = 256$}, $a = 3$, $b = 1$, 
$\{ 102, 54, 48\}$, $\{ 69, 22, 91\}$, $\{ 90, 10, 100\}$, $\{ 123, 107, 16\}$, $\{ 49, 76, 125\}$, $\{ 80, 84, 4\}$, $\{ 95, 37, 58\}$, $\{ 20, 79, 99\}$, $\{ 77, 27, 104\}$, $\{ 112, 62, 50\}$, $\{ 117, 110, 7\}$, $\{ 78, 59, 19\}$, $\{ 45, 60, 105\}$, $\{ 68, 35, 103\}$, $\{ 109, 42, 67\}$, $\{ 94, 25, 119\}$, $\{ 70, 87, 17\}$, $\{ 26, 5, 31\}$, $\{ 64, 32, 96\}$, $\{ 51, 75, 126\}$, $\{ 106, 40, 66\}$, $\{ 130, 56, 74\}$, $\{ 43, 108, 65\}$, $\{ 120, 63, 57\}$, $\{ 121, 88, 33\}$, $\{ 47, 39, 86\}$, $\{ 23, 116, 93\}$, $\{ 114, 13, 127\}$, $\{ 92, 128, 36\}$, $\{ 24, 98, 122\}$, $\{ 71, 85, 14\}$, $\{ 111, 73, 38\}$, $\{ 52, 113, 61\}$, $\{ 55, 101, 46\}$, $\{ 28, 44, 72\}$, $\{ 81, 115, 34\}$, $\{ 82, 53, 29\}$, $\{ 118, 129, 11\}$, $\{ 124, 83, 41\}$, $\{ 9, 21, 30\}$, $\{ 89, 97, 8\}$ 

\noindent
{\boldmath $\alpha = 262$}, $a = 3$, $b = 1$, 
$\{ 36, 58, 94\}$, $\{ 19, 103, 122\}$, $\{ 67, 121, 54\}$, $\{ 37, 91, 128\}$, $\{ 123, 59, 64\}$, $\{ 48, 71, 119\}$, $\{ 83, 90, 7\}$, $\{ 11, 32, 43\}$, $\{ 40, 61, 101\}$, $\{ 75, 124, 49\}$, $\{ 118, 13, 131\}$, $\{ 9, 80, 89\}$, $\{ 46, 108, 62\}$, $\{ 27, 102, 129\}$, $\{ 109, 57, 52\}$, $\{ 85, 110, 25\}$, $\{ 97, 47, 50\}$, $\{ 34, 10, 44\}$, $\{ 111, 29, 82\}$, $\{ 31, 53, 84\}$, $\{ 73, 60, 133\}$, $\{ 41, 76, 117\}$, $\{ 96, 45, 51\}$, $\{ 72, 105, 33\}$, $\{ 35, 79, 114\}$, $\{ 68, 107, 39\}$, $\{ 116, 30, 86\}$, $\{ 16, 65, 81\}$, $\{ 98, 17, 115\}$, $\{ 26, 126, 100\}$, $\{ 88, 24, 112\}$, $\{ 125, 5, 130\}$, $\{ 4, 70, 74\}$, $\{ 113, 93, 20\}$, $\{ 66, 104, 38\}$, $\{ 99, 28, 127\}$, $\{ 106, 14, 120\}$, $\{ 22, 56, 78\}$, $\{ 77, 132, 55\}$, $\{ 87, 8, 95\}$, $\{ 21, 42, 63\}$, $\{ 23, 69, 92\}$.
}

\end{appendices}

\end{document}